\documentclass[a4paper,11pt]{article}

\usepackage[a4paper, left=3.5cm, right=3.5cm, top=2.8cm, bottom=3.5cm, footskip=1.2cm]{geometry}

\usepackage{amsmath, amssymb, amsthm} 
\usepackage{mathtools}

\usepackage{bm}

\usepackage[inline]{enumitem}
\usepackage{multicol}

\newcommand{\argdot}{\mspace{1.5mu} \cdot \mspace{1.5mu}}

\newcommand\dual{^{\prime}}

\DeclareMathOperator{\AC}{AC}

\DeclareMathOperator{\col}{col}
\DeclareMathOperator{\row}{row}

\DeclareMathOperator{\Var}{Var}

\DeclarePairedDelimiter{\abs}{\lvert}{\rvert} 
\DeclarePairedDelimiter{\norm}{\lVert}{\rVert}

\DeclarePairedDelimiter{\rbra}{\lparen}{\rparen}
\DeclarePairedDelimiter{\cbra}{\lbrace}{\rbrace}
\DeclarePairedDelimiter{\sbra}{\lbrack}{\rbrack}

\DeclarePairedDelimiterX{\Set}[2]{\lbrace}{\rbrace}{\mspace{1mu} #1 : #2 \mspace{1mu}}

\newenvironment{ipmatrix}{\begin{pmatrix}}{\end{pmatrix}}
\newenvironment{irowvec}{\begin{ipmatrix}}{\end{ipmatrix}}

\usepackage[
	unicode=true,
	bookmarks=true,
	bookmarkstype=toc,
	bookmarksnumbered=true,
	colorlinks=true,
	]{hyperref}

\theoremstyle{definition}
\newtheorem{definition}{Definition}[section]
\newtheorem{notation}[definition]{Notation}

\theoremstyle{plain}
\newtheorem{theorem}[definition]{Theorem}
\newtheorem{proposition}[definition]{Proposition}
\newtheorem{corollary}[definition]{Corollary}

\theoremstyle{remark}
\newtheorem{remark}[definition]{Remark}

\numberwithin{equation}{section}

\begin{document}

\renewcommand{\thefootnote}{\fnsymbol{footnote}}

\title{Structural Origin of the Hale Bilinear Form: From the Viewpoint of the Green and Lagrange Identities}
\author{Junya Nishiguchi\thanks{Institute for Fundamental Sciences, Faculty of Science and Engineering, Setsunan University, 17-8, Ikeda-Nakamachi, Neyagawa, Osaka, 572-8508, Japan / Mathematical Science Group, Advanced Institute for Materials Research (AIMR), Tohoku University, Katahira 2-1-1, Aoba-ku, Sendai, 980-8577, Japan}
\footnote{E-mail: junya.nishiguchi@setsunan.ac.jp / junya.nishiguchi.2@gmail.com, ORCID: \url{https://orcid.org/0000-0002-0326-2845}}}
\date{}
\maketitle

\begin{abstract}
In the study of linear retarded functional differential equations (RFDEs), the classical bilinear form proposed by J. Hale has long served as an indispensable tool for the geometric theory.
Nevertheless, its original definition suffers from ambiguities, such as the non-commutativity of matrix products and the precise meaning of the involved double integrals, due to its definition motivated primarily by eigenvalue problems.
Here, by establishing the Green and Lagrange identities associated with a given autonomous linear RFDE, we derive a bilinear form $B$, which is none other than the celebrated Hale bilinear form.
This not only resolves the aforementioned issues in the original formulation, but also clarifies its structural origin within the theory of differential equations.
Furthermore, based on the recent definition of the $M^p$-space via the memory measure, we demonstrate that the bilinear form $B$ uniquely extends to the product space $M^{*q} \times M^p$ by dense extension, and show that the Green and Lagrange identities still hold in the sense of mild solutions.
In terms of its potential for extension to non-autonomous systems, neutral functional differential equations, and differential equations with delay and spatial structures, this work opens a new avenue for the adjoint theory of linear functional differential equations.
\bigskip

\bigskip

\end{abstract}

\section{Introduction}

In the theory of linear retarded functional differential equations (RFDEs), a bilinear form established by J. Hale~\cite{Hale_1971_book}, now widely referred to as the \textit{Hale bilinear form} (see \cite{Frasson--VerduynLunel_2003}), serves as a fundamental tool to capture the dynamics of RFDEs.
Specifically, for an autonomous linear RFDE, the corresponding Hale bilinear form is utilized to compute the projection onto the generalized eigenspaces of the infinitesimal generator of the $C_0$-semigroup defined by the initial value problem.
The Hale bilinear form has become an indispensable tool in invariant manifold theory for nonlinear RFDEs as well as in normal form theory for systems involving parameters (see, e.g., \cite{Chow--Mallet-Paret_1977, Faria_Magalhaes_1995_BT, Faria_Magalhaes_1995_Hopf, Guo--Wu_2013}).

As described above, the Hale bilinear form has played a crucial role ranging from the development of theory of linear RFDEs to its applications to nonlinear problems.
Nevertheless, its definition suffers from several ambiguities, such as the non-commutativity of matrix products and the precise meaning of the double integrals involved.
These issues date back to the very introduction of the bilinear form and have remained unresolved even in the standard literature (see \cite{Hale_1977, Hale--VerduynLunel_1993}).
Furthermore, according to Hale, one of the motivations for introducing the bilinear form lies in the eigenvalue problem of an autonomous linear RFDE.
However, this specific motivation led to the introduction of the bilinear form from a purely computational perspective.

In this paper, we show that, for a given autonomous linear RFDE
\begin{equation}\label{eq:RFDE_intro}
	\dot{x}(t) = Lx_t \quad (t \ge 0),
\end{equation}
a bilinear form $B(\psi, \phi)$ naturally arises from the perspective of the \textit{Green identity}, under the duality induced by row and column vectors.
See Section~\ref{sec:Green_identity_bilinear_form} for the definition of the bilinear form $B$, which is none other than the celebrated Hale bilinear form.
Consequently, our framework not only provides a non-\textit{ad-hoc} definition of the Hale bilinear form but also elucidates its structural origin.
Furthermore, it resolves the aforementioned ambiguities present in Hale's formulation.
The Green identity, expressed in terms of the bilinear form $B$, leads to the \textit{Lagrange identity} as a local differential relation.
These considerations naturally yield the adjoint equation
\begin{equation}\label{eq:adjoint_RFDE_intro}
	\dot{y}(t) = -L^*y^t \quad (t \le 0)
\end{equation}
of the RFDE~\eqref{eq:RFDE_intro} (see Section~\ref{sec:adjoint_eq_semigroup} for details).
That is, our approach derives both the Hale bilinear form and the adjoint equation, both of which have traditionally been introduced somewhat heuristically in the existing literature.

For the RFDE~\eqref{eq:RFDE_intro} and its adjoint~\eqref{eq:adjoint_RFDE_intro}, the Green identity naturally implies that the $C_0$-semigroups generated by their respective initial value problems are ``adjoint'' to each other with respect to the bilinear form $B$.
Combining this with the boundedness of the bilinear form $B$, one can see that the infinitesimal generators of these $C_0$-semigroups are also ``adjoint'' to each other with respect to the bilinear form $B$.
This stands in sharp contrast to the classical approach of Hale~\cite{Hale_1971_book}, who initially defined the adjoint operator for the infinitesimal generator of the original equation with respect to his \textit{ad hoc} bilinear form, and subsequently attempted to derive the adjoint equation from it.
For Hale and his collaborators, the adjoint operator and equation obtained in this manner were merely ``formal adjoints'' to the ``true adjoints'' in the sense of functional analysis.
In contrast, our approach based on the Green identity reveals that the bilinear form $B$ itself, as derived from the identity, is the truly essential structure.

Given that the bilinear form $B$, which coincides with the Hale bilinear form, is the truly essential structure in the adjoint theory of the RFDE~\eqref{eq:RFDE_intro}, a natural question arises as to how the notion of mild solutions interacts with its adjoint equation~\eqref{eq:adjoint_RFDE_intro}.
This notion of mild solutions was introduced by Nishiguchi~\cite{Nishiguchi_2023} to provide a framework that allows ``integrable functions" to be selected as initial history functions.

Strictly speaking, the term ``integrable functions'' is somewhat imprecise from the viewpoint of Lebesgue integration theory.
This is because the initial value problem requires that the present value $\phi(0)$ of the initial history function $\phi$ be well-defined, a requirement that cannot be captured by the ``almost everywhere" equivalence relation inherent in the Lebesgue measure.
To accommodate such functions, a space was initially introduced as the $M^p$-space by Delfour and Mitter~\cite{Delfour--Mitter_1972_hereditary}.
However, it remained somewhat unnatural from the perspective of measure theory.
With this motivation, in subsequent papers by Delfour and his collaborators, the product space of an $L^p$-space and a Euclidean space came to be utilized instead.
A direct approach to resolving this unnaturalness was provided by the present author in a recent paper~\cite{Nishiguchi_2026} following \cite{Nishiguchi_2023}.
In that work, the $M^p$-space is redefined as the $L^p$-space with respect to the \textit{memory measure} $m$, which is defined as the sum of the Lebesgue measure on $[-r, 0]$ (where $r > 0$ represents the retardation) and the Dirac measure concentrated at $0$.
We note that similar measures have been considered in different contexts (see, e.g.,  \cite{Coleman--Mizel_1966, Chekroun--Ghil--Liu--Wang_2016}).

The above definition of the $M^p$-space using the memory measure $m$ offers a new perspective on the study of mild solutions to the RFDE~\eqref{eq:RFDE_intro}.
Indeed, we can show that mild solutions to \eqref{eq:RFDE_intro} with initial histories in the $M^p$-space are obtained as a dense extension of those with continuous initial histories.
Here, it is essential that the mild solutions are defined by using the \textit{retarded integral}
\begin{equation*}
	\rbra*{\int_0^t x_s \mspace{2mu} ds}(\theta)
	\coloneqq \int_0^t x(s + \theta) \mspace{2mu} ds
	\quad
	(\theta \in [-r, 0]).
\end{equation*}
See Subsection~\ref{subsec:review_mild_sols} for details of mild solutions to \eqref{eq:RFDE_intro}.
The aforementioned dense extension not only provides an alternative proof for the existence and uniqueness of solutions to the initial value problem associated with the mild solutions to \eqref{eq:RFDE_intro}, but also elucidates the naturalness of considering mild solutions in the study of autonomous linear RFDEs.

By applying the framework of dense extension to the bilinear form $B$, we can show that $B$ uniquely extends to a bounded bilinear form $\tilde{B}$ on the product Banach space
\begin{equation*}
	M^{*q} \times M^p.
\end{equation*}
Here the space $M^{*q}$ emerges as the ``adjoint'' to $M^p$ (see Subsection~\ref{subsec:bounded_extension_B} for details).
This extension naturally induces the notion of mild solutions to the adjoint equation~\eqref{eq:adjoint_RFDE_intro}.
Furthermore, we derive that the $C_0$-semigroup defined by the mild solutions to \eqref{eq:RFDE_intro} and the $C_0$-semigroup defined by those to \eqref{eq:adjoint_RFDE_intro} are ``adjoint'' to each other with respect to the extended bilinear form $\tilde{B}$.
We note that the introduction of a bilinear form on the product space
\begin{equation*}
	\rbra*{L^2([-r, 0], \mathbb{C}^n) \times \mathbb{C}^n} \times \rbra*{L^2([-r, 0], \mathbb{C}^n) \times \mathbb{C}^n}
\end{equation*}
and the elucidation of its connection to the Hale bilinear form were previously addressed by Delfour and Manitius~\cite{Delfour--Manitius_1980}.
Although the adjoint theory for the RFDE~\eqref{eq:RFDE_intro} in the product Hilbert space
\begin{equation*}
	L^2([-r, 0], \mathbb{C}^n) \times \mathbb{C}^n
\end{equation*}
was extensively studied by Delfour, Manitius, and their collaborators (see, e.g., \cite{Bernier--Manitius_1978, Vinter_1978, Delfour_1980, Delfour--Manitius_1980, Manitius_1980, Salamon_1985}), it should be emphasized that neither the formulation of the extended bilinear form $\tilde{B}$ based on the dense extension nor the construction of the adjoint theory using $\tilde{B}$ has been developed in the literature.
For relatively recent studies utilizing the product Hilbert space $L^2([-r, 0], \mathbb{C}^n) \times \mathbb{C}^n$, see \cite{Banks--Robbins--Sutton_2013, Breda--VanVleck_2014, Chekroun--Ghil--Liu--Wang_2016, Chekroun--Liu_2024} for example.

The fact that the $C_0$-semigroup defined by the mild solutions to \eqref{eq:RFDE_intro} and the $C_0$-semigroup defined by those to \eqref{eq:adjoint_RFDE_intro} are ``adjoint'' to each other with respect to the extended bilinear form $\tilde{B}$ strongly suggests that the Green identity and the Lagrange identity also hold in the mild sense.
In this paper, we extend the Lagrange identity to the mild sense and subsequently derive the Green identity from it.
Crucially, the framework of dense extension plays an essential role in this process as well.
\smallskip

This paper is organized as follows.
In Section~\ref{sec:Green_identity_bilinear_form}, we derive the Green identity and the Lagrange identity for the RFDE~\eqref{eq:RFDE_intro}, through which the bilinear form $B$ naturally emerges.
In Section~\ref{sec:adjoint_eq_semigroup}, based on the structural insights from Section~\ref{sec:Green_identity_bilinear_form}, we obtain the adjointness with respect to the bilinear form $B$ between the $C_0$-semigroups generated by the initial value problems for the RFDE~\eqref{eq:RFDE_intro} and its adjoint~\eqref{eq:adjoint_RFDE_intro}.
Motivated by this fact, in the remaining sections (Sections~\ref{sec:mild_sol_semigroup} and \ref{sec:Lagrange_identity_mild}), we investigate the interaction between the notion of mild solutions and the adjoint theory.
In Section~\ref{sec:mild_sol_semigroup}, we revisit the initial value problem for mild solutions to \eqref{eq:RFDE_intro} and reformulate both the unique existence of a solution and the associated $C_0$-semigroup from the perspective of dense extension.
In Section~\ref{sec:Lagrange_identity_mild}, through the dense extension of the bilinear form $B$ to the bounded bilinear form $\tilde{B}$ on $M^{*q} \times M^p$, we examine the adjoint relation with respect to $\tilde{B}$ between the $C_0$-semigroups determined by the mild solutions to the RFDE~\eqref{eq:RFDE_intro} and its adjoint~\eqref{eq:adjoint_RFDE_intro}.
Concluding the section, motivated by this adjointness, we establish the Lagrange identity in the mild sense for the RFDE~\eqref{eq:RFDE_intro}.

\subsection*{Notation}

Throughout this paper, let $r$ be a positive constant and $n$ be a positive integer.
\begin{itemize}
\item Let $\mathbb{C}^n$ be the set of $n$-dimensional complex column vectors, $(\mathbb{C}^n)\dual$ be the set of $n$-dimensional complex row vectors, and $M_n(\mathbb{C})$ be the set of $n \times n$ complex matrices.
Here $\mathbb{C}^n$ and $(\mathbb{C}^n)\dual$ are considered as Banach spaces endowed with the Euclidean norm $\abs*{\argdot}$, whereas $M_n(\mathbb{C})$ is considered as a Banach space endowed with the operator norm $\norm*{\argdot}$.
We denote a row vector $\begin{irowvec} b_1 & \cdots & b_n \end{irowvec} \in (\mathbb{C}^n)\dual$ by $\row(b_1, \dots, b_n)$.
Similarly, a column vector $(a_1, \dots, a_n) \in \mathbb{C}^n$ is also denoted by $\col(a_1, \dots, a_n)$.
\item For a closed and bounded interval $I$ and a Banach space, let $C(I, X)$ be the Banach space of all continuous functions from $I$ to $X$, equipped with the supremum norm $\norm{\argdot}_{C(I)}$.
\item For a closed and bounded interval $I$, let $\AC(I, \mathbb{C}^n)$ be the Banach space of all absolutely continuous functions from $I$ to $\mathbb{C}^n$, equipped with the norm given by
\begin{equation*}
	\norm*{x}_{\AC(I)}
	\coloneqq \norm*{x}_{C(I)} + \norm*{\dot{x}}_{L^1(I)}
\end{equation*}
for $x \in \AC(I, \mathbb{C}^n)$.
Here $\dot{x}$ denotes the derivative of $x \in \AC(I, \mathbb{C}^n)$, which exists almost everywhere on $I$.
\end{itemize}

\section{Green Identity and Bilinear Form}\label{sec:Green_identity_bilinear_form}

Throughout this paper, let $L \colon C([-r, 0], \mathbb{C}^n) \to \mathbb{C}^n$ be a continuous linear map.
We consider the autonomous linear RFDE
\begin{equation}\label{eq:RFDE}
	\dot{x}(t) = Lx_t \quad (t \ge 0)
\end{equation}
for a continuous unknown function $x \colon [-r, \infty) \to \mathbb{C}^n$.
Here $x_t \in C([-r, 0], \mathbb{C}^n)$ is the \textit{history segment} of $x$ at $t \ge 0$ defined by
\begin{equation*}
	x_t(\theta) = x(t + \theta) \quad (\theta \in [-r, 0]).
\end{equation*}

Let $\eta \colon [-r, 0] \to M_n(\mathbb{C})$ be a function of bounded variation such that the operator $L$ is represented as
\begin{equation*}
	L\phi = \int_{-r}^0 d\eta(\theta) \mspace{2mu} \phi(\theta)
\end{equation*}
for all $\phi \in C([-r, 0], \mathbb{C}^n)$.
Then the RFDE~\eqref{eq:RFDE} can be written as
\begin{equation*}
	\dot{x}(t) = \int_{-r}^0 d\eta(\theta) \mspace{2mu} x(t + \theta)
	\quad (t \ge 0).
\end{equation*}

\subsection{Duality and Compact Support}

The natural duality between row and column vectors leads to the following proposition, which concerns row-vector-valued and column-vector-valued continuous functions with compact support.

\begin{proposition}\label{prop:duality_compact_support}
For any continuous functions $u \colon \mathbb{R} \to \mathbb{C}^n$ and $v \colon \mathbb{R} \to (\mathbb{C}^n)\dual$ with compact support,
\begin{equation}\label{eq:duality_compact_support}
	\int_{-\infty}^\infty v(t) \rbra*{\int_{-r}^0 d\eta(\theta) \mspace{2mu} u(t + \theta)} \mspace{2mu} dt
	= \int_{-\infty}^\infty \rbra*{\int_{-r}^0 v(t - \theta) \mspace{2mu} d\eta(\theta)} u(t) \mspace{2mu} dt
\end{equation}
holds.
\end{proposition}

\begin{proof}
We note that the uniform continuity of $u$ and $v$ yields the continuity of
\begin{align*}
	\mathbb{R} \ni t &\mapsto \int_{-r}^0 d\eta(\theta) \mspace{2mu} u(t + \theta) \in \mathbb{C}^n, \\
	\mathbb{R} \ni t &\mapsto \int_{-r}^0 v(t - \theta) \mspace{2mu} d\eta(\theta) \in (\mathbb{C}^n)\dual.
\end{align*}
For each $i, j \in \{1, 2, \dots, n\}$, let
\begin{align*}
	I_{ij} &\coloneqq \int_{-\infty}^\infty v_i(t) \rbra*{\int_{-r}^0 d\eta_{ij}(\theta) \mspace{2mu} u_j(t + \theta)} \mspace{2mu} dt, \\
	J_{ij} &\coloneqq \int_{-\infty}^\infty \rbra*{\int_{-r}^0 v_i(t - \theta) \mspace{2mu} d\eta_{ij}(\theta)} u_j(t) \mspace{2mu} dt,
\end{align*}
where $u = \col(u_1, \dots, u_n)$, $v = \row(v_1, \dots, v_n)$, and $\eta = (\eta_{ij})$.
Then \eqref{eq:duality_compact_support} is expressed as
\begin{equation*}
	\sum_{i, j = 1}^n I_{ij} = \sum_{i, j = 1}^n J_{ij}.
\end{equation*}
We now show that $I_{ij} = J_{ij}$ holds for each $i, j \in \{1, 2, \dots, n\}$.
Since $u$ and $v$ have compact support,
\begin{align*}
	I_{ij} &= \int_{-T}^T v_i(t) \rbra*{\int_{-r}^0 d\eta_{ij}(\theta) \mspace{2mu} u_j(t + \theta)} \mspace{2mu} dt, \\
	J_{ij} &= \int_{-T}^T \rbra*{\int_{-r}^0 v_i(t - \theta) \mspace{2mu} d\eta_{ij}(\theta)} u_j(t) \mspace{2mu} dt
\end{align*}
hold for all sufficiently large $T > 0$.
Let a sufficiently large $T > 0$ be given.
By Fubini's theorem for Stieltjes integrals (see \cite[Theorem~15a in Chapter I]{Widder_1941}), we have
\begin{align*}
	I_{ij}
	&= \int_{-r}^0 \rbra*{\int_{-T}^T v_i(t)u_j(t + \theta) \mspace{2mu} dt} \mspace{2mu} d\eta_{ij}(\theta) \\
	&= \int_{-r}^0 \rbra*{\int_{-T + \theta}^{T + \theta} v_i(t - \theta)u_j(t) \mspace{2mu} dt} \mspace{2mu} d\eta_{ij}(\theta),
\end{align*}
where the last term is equal to
\begin{equation*}
	\int_{-r}^0 \rbra*{\int_{-T'}^{T'} v_i(t - \theta)u_j(t) \mspace{2mu} dt} \mspace{2mu} d\eta_{ij}(\theta)
\end{equation*}
for all sufficiently large $T' > 0$ by choosing a sufficiently large $T > 0$ appropriately since $u$ has compact support.
Applying Fubini's theorem for Stieltjes integrals again, we obtain
\begin{equation*}
	I_{ij}
	= \int_{-T'}^{T'} \rbra*{\int_{-r}^0 v_i(t - \theta) \mspace{2mu} d\eta_{ij}(\theta)} u_j(t) \mspace{2mu} dt
	= J_{ij}.
\end{equation*}
Therefore, \eqref{eq:duality_compact_support} is concluded.
\end{proof}

\subsection{The Green and Lagrange Identities}

We note that \eqref{eq:duality_compact_support} holds since the integration is taken over the whole real line.
By restricting the domain of integration to a closed and bounded interval $[\alpha, \beta]$, we now examine what can be extracted from the difference between the left and the right-hand sides of \eqref{eq:duality_compact_support}.

In the rest of this section, let $[\alpha, \beta]$ be a closed and bounded interval of $\mathbb{R}$.
We use the following notation.

\begin{notation}
Let $u \in C([\alpha - r, \beta], \mathbb{C}^n)$ and $v \in C([\alpha, \beta + r], (\mathbb{C}^n)\dual)$ be given.
For each $t \in [\alpha, \beta]$, let $u_t \colon [-r, 0] \to \mathbb{C}^n$ and $v^t \colon [0, r] \to (\mathbb{C}^n)\dual$ be the continuous functions defined as
\begin{align*}
	u_t(\theta) &\coloneqq u(t + \theta) \quad (\theta \in [-r, 0]), \\
	v^t(\tau) &\coloneqq v(t + \tau) \quad (\tau \in [0, r]).
\end{align*}
\end{notation}

Before stating a result, we clarify the Riemann integrability of
\begin{equation*}
	[-r, 0] \ni \xi \mapsto \int_{-r}^\xi \psi(\xi - \theta) \mspace{2mu} d\eta(\theta) \in (\mathbb{C}^n)\dual
\end{equation*}
for each $\psi \in C([0, r], (\mathbb{C}^n)\dual)$.
This is shown by defining a continuous function $\tilde{\psi} \colon (-\infty, r] \to (\mathbb{C}^n)\dual$ as
\begin{equation*}
	\tilde{\psi}(t) =
	\begin{cases}
		\psi(0) & (t \in (-\infty, 0]), \\
		\psi(t) & (t \in [0, r])
	\end{cases}
\end{equation*}
because in the equality
\begin{equation*}
	\int_{-r}^\xi \psi(\xi - \theta) \mspace{2mu} d\eta(\theta)
	= \int_{-r}^0 \tilde{\psi}(\xi - \theta) \mspace{2mu} d\eta(\theta) - \int_\xi^0 \psi(0) \mspace{2mu} d\eta(\theta),
\end{equation*}
the first term of the right-hand side is continuous with respect to $\xi$, and the second term is of bounded variation with respect to $\xi$.
Here we are using the property that a function of bounded variation on a closed and bounded interval is Riemann integrable, which follows from the integration by parts for Stieltjes integrals (see \cite[Theorem~4b in Chapter~I]{Widder_1941}).
We note that a similar argument has appeared in \cite{Nishiguchi_2023}.

\begin{theorem}\label{thm:duality_finite_interval}
For any $u \in C([\alpha - r, \beta], \mathbb{C}^n)$ and $v \in C([\alpha, \beta + r], (\mathbb{C}^n)\dual)$,
\begin{align}\label{eq:duality_finite_interval}
	&\int_\alpha^\beta \cbra*{\rbra*{\int_{-r}^0 v(t - \theta) \mspace{2mu} d\eta(\theta)} u(t) - v(t) \rbra*{\int_{-r}^0 d\eta(\theta) \mspace{2mu} u(t + \theta)}} \mspace{2mu} dt \notag \\
	&= \sbra*{\int_{-r}^0 \rbra*{\int_{-r}^\xi v^t(\xi - \theta) \mspace{2mu} d\eta(\theta)} u_t(\xi) \mspace{2mu} d\xi}_{t = \alpha}^\beta
\end{align}
holds.
\end{theorem}

\begin{proof}
As in the proof of Proposition~\ref{prop:duality_compact_support}, we proceed componentwise.
Let $i, j \in \{1, 2, \dots, n\}$ be given and let
\begin{align*}
	I &\coloneqq \int_\alpha^\beta v_i(t) \rbra*{\int_{-r}^0 d\eta_{ij}(\theta) \mspace{2mu} u_j(t + \theta)} \mspace{2mu} dt, \\
	J &\coloneqq \int_\alpha^\beta \rbra*{\int_{-r}^0 v_i(t - \theta) \mspace{2mu} d\eta_{ij}(\theta)} u_j(t) \mspace{2mu} dt.
\end{align*}
By Fubini's theorem for Stieltjes integrals, we have
\begin{equation*}
	J
	= \int_{-r}^0 \rbra*{\int_\alpha^\beta v_i(t - \theta)u_j(t) \mspace{2mu} dt} \mspace{2mu} d\eta_{ij}(\theta)
\end{equation*}
and 
\begin{align*}
	I
	&= \int_{-r}^0 \rbra*{\int_\alpha^\beta v_i(t)u_j(t + \theta) \mspace{2mu} dt} \mspace{2mu} d\eta_{ij}(\theta) \\
	&= \int_{-r}^0 \rbra*{\int_{\alpha + \theta}^{\beta + \theta} v_i(t - \theta)u_j(t) \mspace{2mu} dt} \mspace{2mu} d\eta_{ij}(\theta).
\end{align*}
Therefore,
\begin{align*}
	J - I
	&= \int_{-r}^0 \rbra*{\int_{\beta + \theta}^\beta v_i(t - \theta)u_j(t) \mspace{2mu} dt - \int_{\alpha + \theta}^\alpha v_i(t - \theta)u_j(t) \mspace{2mu} dt} \mspace{2mu} d\eta_{ij}(\theta) \\
	&= \int_{-r}^0 \sbra*{\int_{s + \theta}^s v_i(t - \theta)u_j(t) \mspace{2mu} dt}_{s = \alpha}^\beta \mspace{2mu} d\eta_{ij}(\theta)
\end{align*}
holds.
Let $s \in [\alpha, \beta]$ be fixed.
By considering the function $w \colon [\alpha, \beta + r] \to (\mathbb{C}^n)\dual$ defined by
\begin{equation*}
	w(t) =
	\begin{cases}
		v_i(s) & (t \in [\alpha, s]), \\
		v_i(t) & (t \in [s, \beta + r]),
	\end{cases}
\end{equation*}
we have
\begin{align*}
	\int_{s + \theta}^s v_i(t - \theta)u_j(t) \mspace{2mu} dt
	&= \int_{s + \theta}^s w(t - \theta)u_j(t) \mspace{2mu} dt \\
	&= \int_{s - r}^s w(t - \theta)u_j(t) \mspace{2mu} dt - v_i(s)\int_{s - r}^{s + \theta} u_j(t) \mspace{2mu} dt,
\end{align*}
where $t - \theta \le s$ for $t \le s + \theta$ is used.
This yields that the integral $\int_{s + \theta}^s v_i(t - \theta)u_j(t) \mspace{2mu} dt$ is continuous with respect to $\theta \in [-r, 0]$.
Thus, we obtain
\begin{equation*}
	J - I
	= \sbra*{\int_{-r}^0 \rbra*{\int_{s + \theta}^s v_i(t - \theta)u_j(t) \mspace{2mu} dt} \mspace{2mu} d\eta_{ij}(\theta)}_{s = \alpha}^\beta.
\end{equation*}
Let
\begin{equation*}
	K \coloneqq \int_{-r}^0 \rbra*{\int_{s + \theta}^s v_i(t - \theta)u_j(t) \mspace{2mu} dt} \mspace{2mu} d\eta_{ij}(\theta).
\end{equation*}
Then the above discussion indicates that we have
\begin{align*}
	K
	&= \int_{-r}^0 \rbra*{\int_{s - r}^s w(t - \theta)u_j(t)} \mspace{2mu} d\eta_{ij}(\theta) - v_i(s) \int_{-r}^0 \rbra*{\int_{s - r}^{s + \theta} u_j(t) \mspace{2mu} dt} \mspace{2mu} d\eta_{ij}(\theta) \\
	&= \int_{-r}^0 \rbra*{\int_{-r}^0 w(s + \xi - \theta)u_j(s + \xi) \mspace{2mu} d\xi} \mspace{2mu} d\eta_{ij}(\theta) \\
	&\mspace{30mu} - v_i(s) \int_{-r}^0 \rbra*{\int_{-r}^{\theta} u_j(s + \xi) \mspace{2mu} d\xi} \mspace{2mu} d\eta_{ij}(\theta)
\end{align*}
with the change of variable $t = s + \xi$.
Fubini's theorem for Stieltjes integrals yields
\begin{align*}
	&\int_{-r}^0 \rbra*{\int_{-r}^0 w(s + \xi - \theta)u_j(s + \xi) \mspace{2mu} d\xi} \mspace{2mu} d\eta_{ij}(\theta) \\
	&= \int_{-r}^0 \rbra*{\int_{-r}^0 w(s + \xi - \theta) \mspace{2mu} d\eta_{ij}(\theta)} u_j(s + \xi) \mspace{2mu} d\xi,
\end{align*}
and the integration by parts formula for Stieltjes integrals yields
\begin{align*}
	&v_i(s)\int_{-r}^0 \rbra*{\int_{-r}^{\theta} u_j(s + \xi) \mspace{2mu} d\xi} \mspace{2mu} d\eta_{ij}(\theta) \\
	&= v_i(s)\sbra*{\rbra*{\int_{-r}^{\theta} u_j(s + \xi) \mspace{2mu} d\xi}\eta_{ij}(\theta)}_{\theta = -r}^0 - v_i(s)\int_{-r}^0 u_j(s + \theta)\eta_{ij}(\theta) \mspace{2mu} d\theta \\
	&= v_i(s) \int_{-r}^0 u_j(s + \theta)[\eta_{ij}(0) - \eta_{ij}(\theta)] \mspace{2mu} d\theta.
\end{align*}
Thus, 
\begin{align*}
	K
	&= \int_{-r}^0 \rbra*{\int_{-r}^\xi w(s + \xi - \theta) \mspace{2mu} d\eta_{ij}(\theta)} u_j(s + \xi) \mspace{2mu} d\xi \\
	&= \int_{-r}^0 \rbra*{\int_{-r}^\xi v_i(s + \xi - \theta) \mspace{2mu} d\eta_{ij}(\theta)} u_j(s + \xi) \mspace{2mu} d\xi
\end{align*}
is concluded since
\begin{align*}
	&\int_{-r}^0 \rbra*{\int_0^\xi w(s + \xi - \theta) \mspace{2mu} d\eta_{ij}(\theta)} u_j(s + \xi) \mspace{2mu} d\xi \\
	&= v_i(s)\int_{-r}^0 \sbra*{\eta_{ij}(\xi) - \eta_{ij}(0)}u_j(s + \xi) \mspace{2mu} d\xi
\end{align*}
holds in view of $s + \xi - \theta \le s$ for $\xi \le \theta \le 0$.
This shows that \eqref{eq:duality_finite_interval} holds.
\end{proof}

\begin{remark}
The above proof contains a verification of Fubini's theorem for double Stieltjes integrals over a triangular domain, whose argument has been used in \cite{Nishiguchi_2026}.
While this specific triangular formulation is rarely stated explicitly in standard textbooks on classical Stieltjes integrals, it can be rigorously justified by passing to the measure-theoretic framework of Lebesgue--Stieltjes integrals and applying the general Fubini theorem (see, e.g., \cite[Chapter~8]{Rudin_1987}). 
In what follows, this result will be freely utilized.
\end{remark}

Theorem~\ref{thm:duality_finite_interval} leads us to the following results on the Green and Lagrange identities for the RFDE~\eqref{eq:RFDE}, together with the definition of a bilinear form $B$.

Throughout the rest of this paper, we use the following notation.

\begin{notation}
The Banach spaces
\begin{equation*}
	C([-r, 0], \mathbb{C}^n)
	\mspace{10mu} \text{and} \mspace{10mu}
	C([0, r], (\mathbb{C}^n)\dual)
\end{equation*}
are abbreviated as $C$ and $C^*$, respectively.
Accordingly, their norms are denoted by $\norm*{\argdot}_C$ and $\norm*{\argdot}_{C^*}$.
\end{notation}

\begin{definition}\label{dfn:bilinear_form_B}
Let $B \colon C^* \times C \to \mathbb{C}$ be the bilinear form defined by
\begin{equation*}
	B(\psi, \phi)
	= \psi(0)\phi(0) + \int_{-r}^0 \rbra*{\int_{-r}^\xi \psi(\xi - \theta) \mspace{2mu} d\eta(\theta)} \phi(\xi) \mspace{2mu} d\xi
\end{equation*}
for $(\psi, \phi) \in C^* \times C$.
\end{definition}

\begin{theorem}\label{thm:Green_identity}
Let $u \in C([\alpha - r, \beta], \mathbb{C}^n)$ and $v \in C([\alpha, \beta + r], (\mathbb{C}^n)\dual)$ be given.
If the restrictions $u|_{[\alpha, \beta]}$ and $v|_{[\alpha, \beta]}$ are continuously differentiable, then
\begin{align}\label{eq:duality_bilinear_form}
	&\int_\alpha^\beta \cbra*{\rbra*{\dot{v}(t) + \int_{-r}^0 v(t - \theta) \mspace{2mu} d\eta(\theta)} u(t) + v(t) \rbra*{\dot{u}(t) - \int_{-r}^0 d\eta(\theta) \mspace{2mu} u(t + \theta)}} \mspace{2mu} dt \notag \\
	&= \sbra*{B\rbra*{v^t, u_t}}_{t = \alpha}^\beta
\end{align}
holds.
\end{theorem}

\begin{proof}
\eqref{eq:duality_bilinear_form} is obtained from Theorem~\ref{thm:duality_finite_interval} since
\begin{equation*}
	\int_\alpha^\beta \dot{v}(t)u(t) \mspace{2mu} dt + \int_\alpha^\beta v(t)\dot{u}(t) \mspace{2mu} dt
	= \sbra*{v(t)u(t)}_{t = \alpha}^\beta.
\end{equation*}
\end{proof}

\begin{corollary}\label{cor:Lagrange_identity}
Let $u \in C([\alpha - r, \beta], \mathbb{C}^n)$ and $v \in C([\alpha, \beta + r], (\mathbb{C}^n)\dual)$ be given.
If the restrictions $u|_{[\alpha, \beta]}$ and $v|_{[\alpha, \beta]}$ are continuously differentiable, then the function
\begin{equation*}
	[\alpha, \beta] \ni t \mapsto B\rbra*{v^t, u_t} \in \mathbb{C}
\end{equation*}
is continuously differentiable and satisfies
\begin{align*}
	&\frac{d}{dt} \mspace{2mu} B\rbra*{v^t, u_t} \\
	&= \rbra*{\dot{v}(t) + \int_{-r}^0 v(t - \theta) \mspace{2mu} d\eta(\theta)} u(t) + v(t) \rbra*{\dot{u}(t) - \int_{-r}^0 d\eta(\theta) \mspace{2mu} u(t + \theta)}
\end{align*}
for all $t \in [\alpha, \beta]$.
\end{corollary}

\begin{proof}
Since \eqref{eq:duality_bilinear_form} holds for each subinterval $[\alpha', \beta'] \subset [\alpha, \beta]$ by Theorem~\ref{thm:Green_identity}, the conclusion is obtained by applying the fundamental theorem of calculus.
\end{proof}

\begin{corollary}\label{cor:duality_bilinear_form_adjoint}
If $u \in C([\alpha - r, \beta], \mathbb{C}^n)$ and $v \in C([\alpha, \beta + r], (\mathbb{C}^n)\dual)$ are continuously differentiable on $[\alpha, \beta]$ and satisfy
\begin{equation*}
	\dot{u}(t) = \int_{-r}^0 d\eta(\theta) \mspace{2mu} u(t + \theta), \mspace{15mu}
	\dot{v}(t) = -\int_{-r}^0 v(t - \theta) \mspace{2mu} d\eta(\theta)
\end{equation*}
for all $t \in [\alpha, \beta]$, then the function $[\alpha, \beta] \ni t \mapsto B\rbra*{v^t, u_t} \in \mathbb{C}$ is constant.
\end{corollary}

\begin{proof}
This is a consequence of Corollary~\ref{cor:Lagrange_identity}.
\end{proof}

\subsection{Another Expression of the Bilinear Form $B$}

Finally, we provide another expression for the bilinear form $B$.
We note that the function
\begin{equation*}
	[0, r] \ni \zeta \mapsto \int_{-r}^{-\zeta} d\eta(\theta) \mspace{2mu} \phi(\zeta + \theta) \in \mathbb{C}^n
\end{equation*}
is also Riemann integrable for each $\phi \in C$.

\begin{proposition}\label{prop:bilinear_form_B_alternative}
For every $(\psi, \phi) \in C^* \times C$,
\begin{equation*}
	B(\psi, \phi)
	= \psi(0)\phi(0) + \int_0^r \psi(\zeta) \rbra*{\int_{-r}^{-\zeta} d\eta(\theta) \mspace{2mu} \phi(\zeta + \theta)} \mspace{2mu} d\zeta
\end{equation*}
holds.
\end{proposition}

\begin{proof}
Let $\phi = \col(\phi_1, \dots, \phi_n)$, $\psi = \row(\psi_1, \dots, \psi_n)$, and $\eta = (\eta_{ij})$.
We now show that
\begin{equation*}
	\int_{-r}^0 \rbra*{\int_{-r}^\xi \psi_i(\xi - \theta) \mspace{2mu} d\eta_{ij}(\theta)} \phi_j(\xi) \mspace{2mu} d\xi
	= \int_0^r \psi_i(\zeta) \rbra*{\int_{-r}^{-\zeta} d\eta_{ij}(\theta) \mspace{2mu} \phi_j(\zeta + \theta)} \mspace{2mu} d\zeta
\end{equation*}
holds.
By applying Fubini's theorem for double Stieltjes integrals over a triangular domain, we have
\begin{align*}
	\int_{-r}^0 \rbra*{\int_{-r}^\xi \psi_i(\xi - \theta) \mspace{2mu} d\eta_{ij}(\theta)} \phi_j(\xi) \mspace{2mu} d\xi
	&= \int_{-r}^0 \rbra*{\int_{\theta}^0 \psi_i(\xi - \theta)\phi_j(\xi) \mspace{2mu} d\xi} \mspace{2mu} d\eta_{ij}(\theta) \\
	&= \int_{-r}^0 \rbra*{\int_0^{-\theta} \psi_i(\zeta)\phi_j(\zeta + \theta) \mspace{2mu} d\zeta} \mspace{2mu} d\eta_{ij}(\theta)
\end{align*}
with the change of variable $\xi - \theta = \zeta$.
The repeated application of Fubini's theorem for double Stieltjes integrals over a triangular domain yields
\begin{equation*}
	\int_{-r}^0 \rbra*{\int_0^{-\theta} \psi_i(\zeta)\phi_j(\zeta + \theta) \mspace{2mu} d\zeta} \mspace{2mu} d\eta_{ij}(\theta)
	= \int_0^r \psi_i(\zeta) \rbra*{\int_{-r}^{-\zeta} d\eta_{ij}(\theta) \mspace{2mu} \phi_j(\zeta + \theta)} \mspace{2mu} d\zeta.
\end{equation*}
This completes the proof.
\end{proof}

\subsection{Remarks on the Hale Bilinear Form}

\begin{remark}
In the classic monographs~\cite[Section~21]{Hale_1971_book} and \cite[Section~7.3]{Hale_1977} by Hale, a bilinear form $(\cdot, \cdot) \colon C^* \times C \to \mathbb{C}$ was introduced in an \textit{ad hoc} manner as
\begin{equation}\label{eq:Hale_bilinear_form}
	(\psi, \phi)
	\coloneqq \psi(0)\phi(0) - \int_{-r}^0 \int_0^\theta \psi(\xi - \theta) d\eta(\theta) \phi(\xi) d\xi
\end{equation}
for $(\psi, \phi) \in C^* \times C$.
However, the iterated integral in \eqref{eq:Hale_bilinear_form} is mathematically ill-defined regarding the order of integration and the interpretation of $d\eta(\theta)$.
Specifically, as written, the inner integral is with respect to $\xi$, yet the integrand contains $d\eta(\theta)$ associated with the outer integration variable $\theta$.
Since the product $\psi(\xi - \theta) d\eta(\theta) \phi(\xi)$ cannot be reordered or decoupled freely, the integration cannot be rigorously justified in the framework of standard Stieltjes integrals.
Combining Definition~\ref{dfn:bilinear_form_B}, which is based on Theorem~\ref{thm:duality_finite_interval}, and Proposition~\ref{prop:bilinear_form_B_alternative} solves this ambiguity by providing a rigorous representation via a proper application of Fubini's theorem.
\end{remark}

\begin{remark}
In \cite[Section~7.5]{Hale--VerduynLunel_1993}, a bilinear form $(\cdot, \cdot) \colon C^* \times C \to \mathbb{C}$ was given by
\begin{align*}
	(\psi, \phi)
	&\coloneqq \psi(0)\phi(0) - \int_{-r}^0 \int_{r}^\theta \psi(\theta - \tau) d\eta(\tau) \phi(\theta) d\theta \\
	&= \psi(0)\phi(0) - \int_{-r}^0 \int_0^\theta \psi(\theta - \tau) d\eta(\tau) \phi(\theta) d\theta
\end{align*}
for $(\psi, \phi) \in C^* \times C$, motivated by the adjoint theory in the sense of functional analysis.
While this 1993 formulation shifts the notation of variables, it still inherits a similar structural ambiguity regarding the rigorous justification of the iterated Stieltjes integrals.
For related literature, we refer the reader to \cite[Appendix~A]{Frasson--VerduynLunel_2003}.
\end{remark}

\begin{remark}
In \cite[Subsection~2.2.5]{Guo--Wu_2013}, the following assertion is made without proof: For continuously differentiable functions $u \colon \mathbb{R} \to \mathbb{C}^n$ and $v \colon \mathbb{R} \to (\mathbb{C}^n)\dual$,
\begin{align*}
	&\rbra*{\dot{\bar{v}}(t) + \int_{-r}^0 \bar{v}(t - \theta) \mspace{2mu} d\eta(\theta)} u(t) + \bar{v}(t) \rbra*{\dot{u}(t) - \int_{-r}^0 d\eta(\theta) \mspace{2mu} u(t + \theta)} \\
	&= \frac{d}{dt} \mspace{2mu} \langle v^t, u_t \rangle
\end{align*}
holds.
Here $\langle \cdot, \cdot \rangle \colon C^* \times C \to \mathbb{C}$ is given by
\begin{equation*}
	\langle \psi, \phi \rangle
	\coloneqq \bar{\psi}(0)\phi(0) - \int_{-r}^0 \int_0^\theta \bar{\psi}(\xi - \theta) d\eta(\theta) \phi(\xi) d\xi
\end{equation*}
for $(\psi, \phi) \in C^* \times C$.
It should be noted that, up to complex conjugation, this assertion is none other than the Lagrange identity for the RFDE~\eqref{eq:RFDE}.
We emphasize that Corollary~\ref{cor:Lagrange_identity} establishes this identity in a formulation where any ambiguity regarding the iterated integrals is eliminated.
\end{remark}

\section{Adjoint Equation and Semigroup}\label{sec:adjoint_eq_semigroup}

In this section, we consider the initial value problem (IVP)
\begin{equation}\label{eq:IVP_RFDE}
	\left\{
	\begin{aligned}
		\dot{x}(t) &= Lx_t & (t \ge 0), \\
		x_0 &= \phi
	\end{aligned}
	\right.
\end{equation}
for the RFDE~\eqref{eq:RFDE}, where $\phi \in C$.
It is well known that \eqref{eq:IVP_RFDE} has a unique solution, and the family $(T(t))_{t \ge 0}$ of bounded linear operators on $C$ defined by
\begin{equation*}
	T(t)\phi \coloneqq x(\argdot; \phi)_t
	\quad ((t, \phi) \in [0, \infty) \times C)
\end{equation*}
constitutes a $C_0$-semigroup (see \cite[Lemma 1.2 in Chapter 7]{Hale--VerduynLunel_1993}).
Here
\begin{equation*}
	x(\argdot; \phi) \colon [-r, \infty) \to \mathbb{C}^n
\end{equation*}
denotes the unique solution of \eqref{eq:IVP_RFDE}.

\subsection{Adjoint Equation}

Based on Theorem~\ref{thm:Green_identity} and Corollaries~\ref{cor:Lagrange_identity}, \ref{cor:duality_bilinear_form_adjoint}, we now introduce the following definition.

\begin{definition}
\mbox{}
\begin{itemize}
\item We define a continuous linear map $L^* \colon C^* \to (\mathbb{C}^n)\dual$ by
\begin{equation*}
	L^*\psi = \int_{-r}^0 \psi(-\theta) \mspace{2mu} d\eta(\theta)
\end{equation*}
for $\psi \in C^*$.
\item We call the differential equation
\begin{equation}\label{eq:adjoint_RFDE}
	\dot{y}(t) = -L^*y^t \quad (t \le 0)
\end{equation}
the \textit{adjoint equation} of \eqref{eq:RFDE}.
Here, for a continuous function $y \colon (-\infty, r] \to (\mathbb{C}^n)\dual$, we denote by $y^t \in C^*$ the \textit{future segment} of $y$ at $t \le 0$, which is defined by
\begin{equation*}
	y^t(\tau) = y(t + \tau) \quad (\tau \in [0, r]).
\end{equation*}
\end{itemize}
\end{definition}

For the adjoint equation~\eqref{eq:adjoint_RFDE}, we also consider its IVP
\begin{equation}\label{eq:IVP_adjoint_eq}
	\left\{
	\begin{aligned}
		\dot{y}(t) &= -L^*y^t & (t \le 0), \\
		y^0 &= \psi,
	\end{aligned}
	\right.
\end{equation}
where $\psi \in C^*$.

By considering the function $z \colon [-r, \infty) \to \mathbb{C}^n$ defined by $z(t) = y(-t)^T$ for $t \in [-r, \infty)$, \eqref{eq:adjoint_RFDE} can be written as
\begin{equation*}
	\dot{z}(t) = \int_{-r}^0 d\eta(\theta)^T \mspace{2mu} z(t + \theta) \quad (t \ge 0).
\end{equation*}
Here $\argdot^T$ denotes the matrix transpose.
Therefore, \eqref{eq:IVP_adjoint_eq} has a unique solution, which we denote by
\begin{equation*}
	y(\argdot; \psi) \colon (-\infty, r] \to (\mathbb{C}^n)\dual.
\end{equation*}
It also holds that the family $(T^*(t))_{t \ge 0}$ of linear operators on $C^*$ defined by
\begin{equation*}
	T^*(t)\psi \coloneqq y(\argdot; \psi)^{-t}
	\quad ((t, \psi) \in [0, \infty) \times C^*)
\end{equation*}
is a $C_0$-semigroup.

\subsection{Adjoint Relation with respect to the Bilinear Form $B$}

The following theorem establishes that the $C_0$-semigroup $(T^*(t))_{t \ge 0}$ is ``adjoint" to $(T(t))_{t \ge 0}$ with respect to the bilinear form $B$.

\begin{theorem}\label{thm:adjoint_semigroup}
For every $(\psi, \phi) \in C^* \times C$,
\begin{equation*}
	B(\psi, T(t)\phi) = B(T^*(t)\psi, \phi) \quad (t \ge 0)
\end{equation*}
holds.
\end{theorem}

\begin{proof}
Let $\beta > 0$ be fixed, and let $x \coloneqq x(\argdot; \phi)|_{[0, \beta]}$ and $y \coloneqq y(\argdot; \psi)|_{[-\beta, 0]}$.
Since the function $\tilde{y} \colon [0, \beta] \to (\mathbb{C}^n)\dual$ defined by $\tilde{y}(t) = y(t - \beta)$ for $t \in [0, \beta]$ satisfies
\begin{equation*}
	\frac{d}{dt} \mspace{2mu} \tilde{y}(t) = -L^*\tilde{y}^t
	\quad (t \in [0, \beta]),
\end{equation*}
Corollary~\ref{cor:duality_bilinear_form_adjoint} yields
\begin{equation*}
	B\rbra[\big]{\tilde{y}^\beta, x_\beta} = B\rbra*{\tilde{y}^0, x_0}.
\end{equation*}
This implies $B(\psi, T(\beta)\phi) = B(T^*(\beta)\psi, \phi)$ because we have
\begin{equation*}
	\tilde{y}^\beta = y^0 = \psi, \mspace{15mu} \tilde{y}^0 = y^{-\beta} = T^*(\beta)\psi.
\end{equation*}
Since $\beta > 0$ is arbitrary, this completes the proof.
\end{proof}

It follows from Definition~\ref{dfn:bilinear_form_B} that the bilinear form $B \colon C^* \times C \to \mathbb{C}$ is bounded, i.e., there is a constant $M > 0$ such that
\begin{equation*}
	\abs*{B(\psi, \phi)}
	\le M\norm*{\psi}_{C^*}\norm*{\phi}_{C}
\end{equation*}
holds for all $(\psi, \phi) \in C^* \times C$.
By combining this boundedness and Theorem~\ref{thm:adjoint_semigroup}, we obtain the following corollary.

\begin{corollary}
Let $A$ and $A^*$ be the infinitesimal generators of the $C_0$-semigroups $(T(t))_{t \ge 0}$ and $(T^*(t))_{t \ge 0}$, respectively. Then
\begin{equation*}
	B(\psi, A\phi) = B(A^*\psi, \phi)
\end{equation*}
holds for all $\phi \in \mathcal{D}(A)$ and $\psi \in \mathcal{D}(A^*)$, where $\mathcal{D}(A)$ and $\mathcal{D}(A^*)$ denote the domains of $A$ and $A^*$, respectively.
\end{corollary}

\begin{proof}
Since $B$ is bilinear, Theorem~\ref{thm:adjoint_semigroup} implies that 
\begin{equation*}
	B\rbra*{\psi, \frac{1}{t}(T(t)\phi - \phi)}
	= B\rbra*{\frac{1}{t}(T^*(t)\psi - \psi), \phi}
\end{equation*}
holds for all $(\psi, \phi) \in C^* \times C$ and $t > 0$.
Then the conclusion is obtained by taking the limit as $t \to 0$ in the above equation, in view of the boundedness of $B$.
\end{proof}

\section{Mild Solution and Semigroup}\label{sec:mild_sol_semigroup}

\subsection{Review of Mild Solutions}\label{subsec:review_mild_sols}

We first briefly review the notion of mild solutions introduced in \cite{Nishiguchi_2023, Nishiguchi_2026}.

Let $\mathcal{L}^1_\mathrm{loc}([-r, \infty), \mathbb{C}^n)$ denote the set of all $\mathbb{C}^n$-valued locally integrable functions on $[-r, \infty)$.
In this paper, integrable functions are understood to be defined almost everywhere with respect to the specified measure.

A function $x \in \mathcal{L}^1_\mathrm{loc}([-r, \infty), \mathbb{C}^n)$ is said to be a \textit{mild solution} to \eqref{eq:RFDE} if (i) $x$ is defined on $[0, \infty)$, and (ii)
\begin{equation*}
	x(t) = x(0) + L\int_0^t x_s \mspace{2mu} ds
\end{equation*}
holds for all $t \ge 0$.
Here, for each $t \ge 0$, $\int_0^t x_s \mspace{2mu} ds \in C$ is defined as
\begin{equation*}
	\rbra*{\int_0^t x_s \mspace{2mu} ds}(\theta) = \int_0^t x(s + \theta) \mspace{2mu} ds
	\quad (\theta \in [-r, 0]),
\end{equation*}
which is referred to as the \textit{retarded integral} of $x$ over $[0, t]$ following \cite{Nishiguchi_2026}.

\begin{remark}
For $x \in \mathcal{L}^1_\mathrm{loc}([-r, \infty), \mathbb{C}^n)$, it holds that the function
\begin{equation*}
	[0, \infty) \ni t \mapsto \int_0^t x_s \mspace{2mu} ds \in C
\end{equation*}
is continuous (see \cite{Nishiguchi_2023}).
Therefore, every mild solution to \eqref{eq:RFDE} is continuous on $[0, \infty)$.
\end{remark}

\begin{remark}\label{rmk:usual_mild}
When $x \colon [-r, \infty) \to \mathbb{C}^n$ is continuous, $\int_0^t x_s \mspace{2mu} ds$ coincides with the Riemann integral of the vector-valued continuous function $[0, t] \ni s \mapsto x_s \in C$ (see \cite{Nishiguchi_2023}).
Therefore,
\begin{equation}\label{eq:retarded_Riemann}
	L\int_0^t x_s \mspace{2mu} ds = \int_0^t Lx_s \mspace{2mu} ds
\end{equation}
holds in this case.
The above equality is also directly proved by applying Fubini's theorem for Stieltjes integrals in view of
\begin{equation*}
	\int_{-r}^0 d\eta(\theta) \rbra*{\int_0^t x(s + \theta) \mspace{2mu} ds}
	= \int_0^t \rbra*{\int_{-r}^0 d\eta(\theta) \mspace{2mu} x(s + \theta)} \mspace{2mu} ds.
\end{equation*}
The combination of the notion of mild solutions and \eqref{eq:retarded_Riemann} yields that usual solutions of \eqref{eq:RFDE} are mild solutions.
\end{remark}

Throughout the rest of this paper, we will use the following notations.

\begin{notation}
Let $m$ denote the sum of the Lebesgue measure on $[-r, 0]$ and the Dirac measure at $0$.
Following \cite{Nishiguchi_2026}, we call $m$ the \textit{memory measure} on $[-r, 0]$.
\end{notation}

\begin{notation}
For each $1 \le p \le \infty$, we define the Banach space $M^p([-r, 0], \mathbb{C}^n)$ as
\begin{equation*}
	M^p([-r, 0], \mathbb{C}^n) \coloneqq L^p(([-r, 0], m), \mathbb{C}^n)
\end{equation*}
endowed with the norm
\begin{equation*}
	\norm*{\argdot}_{M^p} \coloneqq \norm*{\argdot}_{L^p([-r, 0], m)}.
\end{equation*}
Here $([-r, 0], m)$ denotes the measure space $[-r, 0]$ equipped with the measure $m$.
For simplicity, $M^p([-r, 0], \mathbb{C}^n)$ will be abbreviated as $M^p$.
\end{notation}

\begin{remark}
As is established in \cite{Nishiguchi_2026}, the above measure-theoretic formulation provides a clear framework for RFDEs.
It is worth noting that while Delfour and Mitter~\cite{Delfour--Mitter_1972_hereditary} originally defined the $M^p$-space for $1 \le p < \infty$ by using a seminorm containing the point evaluation $\abs*{\phi(0)}$, the absence of an underlying measure rendered the functional-analytic framework intricate.
To circumvent this complexity, subsequent studies (e.g., \cite{Webb_1976, Bernier--Manitius_1978, Delfour_1980, Delfour--Manitius_1980, Manitius_1980, Salamon_1985}) adopted the product space 
\begin{equation*}
	L^p([-r, 0], \mathbb{C}^n) \times \mathbb{C}^n.
\end{equation*}
In contrast, our memory-measure approach unifies these perspectives into a single $L^p$-space without separating the boundary value from the history.
See also \cite{Nishiguchi_2026} for the above discussion.
\end{remark}

Mild solutions to \eqref{eq:RFDE} can be formulated as an IVP by imposing an initial condition of the form
\begin{equation*}
	x(\theta) = \phi(\theta) \quad (\text{$m$-a.e.\ $\theta \in [-r, 0]$}),
\end{equation*}
where $\phi \in M^1$.
We write the above initial condition as $x_0 = \phi$.
Then the IVP is expressed as
\begin{equation}\label{eq:IVP_RFDE_mild}
	\left\{
	\begin{aligned}
		x(t) &= \phi(0) + L\int_0^t x_s \mspace{2mu} ds & (t \ge 0), \\
		x_0 &= \phi
	\end{aligned}
	\right.
\end{equation}
where $\phi \in M^1$.

\subsection{Existence and Uniqueness: Revisited}

On the above IVP~\eqref{eq:IVP_RFDE_mild}, we recall the following result established by \cite{Nishiguchi_2023}.

\begin{theorem}[\cite{Nishiguchi_2023}]\label{thm:unique_existence}
For any $\phi \in M^1$, \eqref{eq:IVP_RFDE_mild} has a unique solution.
\end{theorem}

In \cite{Nishiguchi_2023}, the direct proof of Theorem~\ref{thm:unique_existence} is given, which is based on the fixed point theorem.
Using the density of $C$ in $M^1$ (see Proposition~\ref{prop:density_C_Mp} given below), an alternative proof of Theorem~\ref{thm:unique_existence} is also possible by assuming the existence and uniqueness of solutions to \eqref{eq:IVP_RFDE}.

\begin{proposition}\label{prop:density_C_Mp}
For each $1 \le p < \infty$, the subset $C$ is dense in $M^p$.
\end{proposition}

\begin{proof}
Let $\phi \in M^p$ be given.
We can choose a sequence $(s_j)_{j = 1}^\infty$ of simple functions from $[-r, 0]$ to $\mathbb{C}^n$ such that $s_j \to \phi$ pointwise $m$-a.e., and $\norm*{\phi - s_j}_{M^p} \to 0$.
We note that for every $\varepsilon > 0$ and $x \in \mathbb{C}^n$, there is a $\chi \in C$ such that
\begin{equation*}
	\chi(0) = x
	\mspace{10mu} \text{and} \mspace{10mu}
	\norm*{\bm{1}_{\{0\}}x - \chi}_{L^p[-r, 0]} < \varepsilon,
\end{equation*}
where $\rbra*{\bm{1}_{\{0\}}x}(\theta) \coloneqq 0$ for $\theta \in [-r, 0)$, and $\rbra*{\bm{1}_{\{0\}}x}(\theta) \coloneqq x$ for $\theta = 0$.
By utilizing this fact, for each $j$, we can also choose a $\chi_j \in C$ such that
\begin{equation*}
	\chi_j(0) = s_j(0)
	\mspace{10mu} \text{and} \mspace{10mu}
	\norm*{s_j - \chi_j}_{L^p[-r, 0]} < \frac{1}{j}.
\end{equation*}
Since this implies $\norm*{s_j - \chi_j}_{M^p} < 1/j$, we have
\begin{equation*}
	\norm*{\phi - \chi_j}_{M^p}
	\le \norm*{\phi - s_j}_{M^p} + \norm*{s_j - \chi_j}_{M^p}
	\to 0
\end{equation*}
as $j \to \infty$.
This completes the proof.
\end{proof}

\begin{proof}[Another proof of Theorem~\ref{thm:unique_existence}]
\textbf{Step 1: Preliminary.}
Let $\phi \in C$ and $t \ge 0$ be given.
From \eqref{eq:retarded_Riemann}, the (usual) solution $x(t; \phi)$ satisfies
\begin{equation*}
	x(t; \phi) = \phi(0) + L\int_0^t x(\argdot; \phi)_s \mspace{2mu} ds.
\end{equation*}
Since 
\begin{align*}
	\norm*{\int_0^t x(\argdot; \phi)_s \mspace{2mu} ds}_{C}
	&\le \sup_{\theta \in [-r, 0]} \int_0^t \abs*{x(s + \theta; \phi)} \mspace{2mu} ds \\
	&\le \int_{-r}^0 \abs{x(s; \phi)} \mspace{2mu} ds + \int_0^t \abs*{x(s; \phi)} \mspace{2mu} ds,
\end{align*}
we have
\begin{align*}
	\abs*{x(t; \phi)}
	&\le \abs*{\phi(0)} + \norm*{L}\rbra*{\int_{-r}^0 \abs*{\phi(\theta)} \mspace{2mu} d\theta + \int_0^t \abs*{x(s; \phi)} \mspace{2mu} ds} \\
	&\le \max\{1, \norm*{L}\}\norm*{\phi}_{M^1} + \norm*{L}\int_0^t \abs*{x(s; \phi)} \mspace{2mu} ds.
\end{align*}
Here $\norm*{L}$ denotes the operator norm of $L$.
Since the above inequality holds for all $t \ge 0$,
\begin{equation}\label{eq:continuity_phi}
	\abs*{x(t; \phi)}
	\le \max\{1, \norm*{L}\}\norm*{\phi}_{M^1} e^{\norm*{L}t}
	\quad (t \ge 0)
\end{equation}
is obtained by applying Gronwall's inequality.

Let $T > 0$ be fixed.
\eqref{eq:continuity_phi} implies that the linear map
\begin{equation*}
	C \ni \phi \mapsto x(\argdot; \phi)|_{[0, T]} \in C([0, T], \mathbb{C}^n)
\end{equation*}
is continuous with respect to $\norm{\argdot}_{M^1}$.
Since the subset $C$ is dense in $M^1$, the above map uniquely extends to the continuous linear map $M^1 \ni \phi \mapsto x(\argdot; \phi)|_{[0, T]} \in C([0, T], \mathbb{C}^n)$.
By applying the above argument for each $T > 0$, one can show that there is a unique linear map $\mathcal{T} \colon M^1 \to C([0, \infty), \mathbb{C}^n)$ such that
\begin{equation*}
	\mathcal{T}\phi = x(\argdot; \phi)|_{[0, \infty)}
\end{equation*}
holds for all $\phi \in C$.
For each $\phi \in M^1$, let $x(\argdot; \phi)$ be the element of $\mathcal{L}^1_\mathrm{loc}([-r, \infty), \mathbb{C}^n)$ defined by $x(\argdot; \phi)|_{[0, \infty)} \coloneqq \mathcal{T}\phi$ and $x(\theta; \phi) \coloneqq \phi(\theta)$ for $m$-a.e.\ $\theta \in [-r, 0]$.
\smallskip

\textbf{Step 2: Existence.}
Let $\phi \in M^1$ be given.
We choose a sequence $(\phi_j)_{j = 1}^\infty$ in $C$ converging to $\phi$ in $M^1$.
For every $j$, we have
\begin{equation*}
	x(t; \phi_j) = \phi_j(0) + L\int_0^t x(\argdot; \phi_j)_s \mspace{2mu} ds
	\quad (t \ge 0)
\end{equation*}
from \eqref{eq:retarded_Riemann}.
Here $\phi_j(0) \to \phi(0)$, and $x(t; \phi_j) \to x(t; \phi)$ for $t \ge 0$.
Furthermore, we have
\begin{align*}
	&\norm*{\int_0^t x(\argdot; \phi)_s \mspace{2mu} ds - \int_0^t x(\argdot; \phi_j)_s \mspace{2mu} ds}_{C} \\
	&\le \int_{-r}^0 \abs*{\phi(\theta) - \phi_j(\theta)} \mspace{2mu} d\theta + \int_0^t \abs*{x(s; \phi) - x(s; \phi_j)} \mspace{2mu} ds \\
	&\le \norm*{\phi - \phi_j}_{M^1} + t \cdot \sup_{s \in [0, t]} \abs*{x(s; \phi) - x(s; \phi_j)},
\end{align*}
which implies that
\begin{equation*}
	L\int_0^t x(\argdot; \phi_j)_s \mspace{2mu} ds \to L\int_0^t x(\argdot; \phi)_s \mspace{2mu} ds
\end{equation*}
by the continuity of $L$.
Therefore, we obtain
\begin{equation*}
	x(t; \phi) = \phi(0) + L\int_0^t x(\argdot; \phi)_s \mspace{2mu} ds
	\quad (t \ge 0),
\end{equation*}
which means that $x(\argdot; \phi)$ is a solution of \eqref{eq:IVP_RFDE_mild}.
\smallskip

\textbf{Step 3: Uniqueness.}
We only have to consider the case $\phi(\theta) = 0$ ($m$-a.e.\ $\theta \in [-r, 0]$) in IVP~\eqref{eq:IVP_RFDE_mild}.
If $x$ is a mild solution to \eqref{eq:RFDE} under $x_0 = 0$, Step 1 yields
\begin{equation*}
	\abs*{x(t)} \le \norm*{L} \int_0^t \abs*{x(s)} \mspace{2mu} ds
	\quad (t \ge 0).
\end{equation*}
Since $x|_{[0, \infty)}$ is continuous, $x(t) = 0$ for $t \ge 0$ is concluded by applying Gronwall's inequality.
\smallskip

This completes the proof.
\end{proof}

\begin{remark}
Under our settings, one of the results obtained by Webb~\cite{Webb_1976} can be stated as follows.
Let $1 \le p < \infty$ be given.
For RFDE~\eqref{eq:RFDE}, let $A$ be the densely defined linear operator on $X \coloneqq L^p([-r, 0], \mathbb{C}^n) \times \mathbb{C}^n$ given by
\begin{align*}
	\mathcal{D}(A)
	&\coloneqq \Set*{(\phi, h) \in X}{\text{$\phi \in \AC([-r, 0], \mathbb{C}^n)$, $\phi' \in L^p([-r, 0], \mathbb{C}^n)$, and $h = \phi(0)$}}, \\
	A(\phi, h)
	&\coloneqq (\phi', L\phi).
\end{align*}
By showing that $A$ generates a $C_0$-semigroup $(T(t))_{t \ge 0}$, Webb~\cite{Webb_1976} defined the function $x(\phi, h)$ for $(\phi, h) \in X$ by
\begin{alignat*}{2}
	x(\phi, h)(t) &\coloneqq \phi(t) &\quad & (a.e.\ t \in [-r, 0)), \\
	x(\phi, h)(t) &\coloneqq P_2T(t)(\phi, h) & & (t \ge 0).
\end{alignat*}
Here $P_2 \colon X \to \mathbb{C}^n$ denotes the second projection.
As stated in \cite{Webb_1976}, this method is an adaptation of the one in the Banach space $C$ by Flaschka and Leitman~\cite{Flaschka--Leitman_1975} to the setting of $X$.
By this semigroup approach, the integral equation
\begin{equation*}
	x(\phi, h)(t)
	= h + \int_{-r}^0 d\eta(\theta) \mspace{2mu} \rbra*{\int_0^t x(\phi, h)(s + \theta) \mspace{2mu} ds}
	\quad (t \ge 0)
\end{equation*}
is shown to hold for all $(\phi, h) \in X$ by virtue of the density of $\mathcal{D}(A)$ in $X$.
In this context, our alternative proof of Theorem~\ref{thm:unique_existence} can be viewed as a direct counterpart to Webb's formulation, which bypasses the need for semigroup generation on the product space.
\end{remark}

For each $\phi \in M^1$, let $x(\argdot; \phi)$ be the unique solution of \eqref{eq:IVP_RFDE_mild}, whose existence is ensured by Theorem~\ref{thm:unique_existence}.
This notation is consistent with the symbol $x(\argdot; \phi)$ for $\phi \in C$ from Remark~\ref{rmk:usual_mild}.

\subsection{Mild Solution and $C_0$-Semigroup}

In this subsection, we discuss a family of linear maps obtained from the mild solutions to \eqref{eq:RFDE}.
Throughout this subsection, $1 \le p < \infty$ is assumed to be fixed.

\begin{definition}
We define the family $(S(t))_{t \ge 0}$ of linear maps on $M^p$ by
\begin{equation*}
	S(t)\phi \coloneqq x(\argdot; \phi)_t
\end{equation*}
for $(t, \phi) \in [0, \infty) \times M^p$.
Here $x(\argdot; \phi)_t \in M^p$ is defined as
\begin{equation*}
	x(\argdot; \phi)_t(\theta) \coloneqq x(t + \theta; \phi)
\end{equation*}
for $m$-a.e.\ $\theta \in [-r, 0]$.
\end{definition}

Regarding the strong continuity of $(S(t))_{t \ge 0}$, it holds that the function $[0, \infty) \ni t \mapsto x(\argdot; \phi)_t \in M^p$ is continuous because
\begin{align*}
	&\norm*{x(\argdot; \phi)_t - x(\argdot; \phi)_s}_{M^p}^p \\
	&= \int_{-r}^0 \abs*{x(t + \theta; \phi) - x(s + \theta; \phi)}^p \mspace{2mu} d\theta + \abs*{x(t; \phi) - x(s; \phi)}^p 
\end{align*}
holds for all $t, s \ge 0$.
We note that this property was used in \cite{Bernier--Manitius_1978}.
Taking this strong continuity into account, we arrive at the following result.

\begin{theorem}\label{thm:mild_dynamics}
$(S(t))_{t \ge 0}$ is a $C_0$-semigroup on $M^p$.
\end{theorem}

\begin{proof}
\textbf{Step 1: Semigroup property.}
Let $\phi \in M^p$ and $t_1, t_2 \ge 0$ be fixed.
We now show that
\begin{equation}\label{eq:semigroup_property}
	x(t_1 + t_2 + \theta; \phi) = x(t_2 + \theta; x(\argdot; \phi)_{t_1})
\end{equation}
holds for $m$-a.e.\ $\theta \in [-r, 0]$.
For this purpose, let $\psi \coloneqq x(\argdot; \phi)_{t_1}$, and let
\begin{equation*}
	y(t) \coloneqq x(t_1 + t; \phi)
\end{equation*}
for all $t \ge 0$ and $m$-a.e.\ $t \in [-r, 0]$.
Then for all $t \ge 0$, we have
\begin{align*}
	y(t)
	&= \phi(0) + L\int_0^{t_1 + t} x(\argdot; \phi)_s \mspace{2mu} ds \\
	&= \phi(0) + L\int_0^{t_1} x(\argdot; \phi)_s \mspace{2mu} ds + L\int_{t_1}^{t_1 + t} x(\argdot; \phi)_s \mspace{2mu} ds \\
	&= x(t_1; \phi) + L\int_0^t x(t_1 + \argdot; \phi)_s \mspace{2mu} ds \\
	&= \psi(0) + L\int_0^t y_s \mspace{2mu} ds.
\end{align*}
Since $y_0 = \psi$,
\begin{equation*}
	y(t) = x(t; \psi) \quad (t \ge 0)
\end{equation*}
holds from Theorem~\ref{thm:unique_existence}.
This shows that the equality~\eqref{eq:semigroup_property} holds.
Therefore, $S(t_1 + t_2)\phi = S(t_2)S(t_1)\phi$ is obtained.
\smallskip

\textbf{Step 2: Boundedness.}
Let $\phi \in M^p$ be given.
Let $q$ be the conjugate exponent of $p$.
Applying H\"older's inequality, we have
\begin{align*}
	\norm*{\phi}_{M^1}
	&= \int_{[-r, 0]} \abs*{\phi(\theta)} \mspace{2mu} m(d\theta) \\
	&\le \rbra*{\int_{[-r, 0]} m(d\theta)}^{1/q} \cdot \rbra*{\int_{[-r, 0]} \abs*{\phi(\theta)}^p \mspace{2mu} m(d\theta)}^{1/p} \\
	&= (r + 1)^{1/q} \norm*{\phi}_{M^p},
\end{align*}
where the integration with respect to the measure $m$ is denoted by $m(d\theta)$.
Therefore, the above proof of Theorem~\ref{thm:unique_existence} shows that there is an $R > 0$ such that
\begin{equation*}
	\abs*{x(t; \phi)} \le Re^{\norm*{L}t}\norm*{\phi}_{M^p}
	\quad (t \ge 0)
\end{equation*}
holds.
From this inequality, we obtain
\begin{align*}
	\norm*{S(t)\phi}_{M^p}^p
	&= \int_{[-r, 0]} \abs*{x(t + \theta; \phi)}^p \mspace{2mu} m(d\theta) \\
	&\le \int_{-r}^0 \abs*{\phi(\theta)}^p \mspace{2mu} d\theta + \int_0^t \abs*{x(s; \phi)}^p \mspace{2mu} ds + \abs*{x(t; \phi)}^p \\
	&\le \norm*{\phi}_{M^p}^p + (t + 1) \cdot \rbra[\big]{Re^{\norm*{L}t}}^p \cdot \norm*{\phi}_{M^p}^p
\end{align*}
for each $t \ge 0$, which yields the boundedness of $S(t)$.
\smallskip

This completes the proof.
\end{proof}

The next result is a consequence of Theorem~\ref{thm:mild_dynamics}.
However, it is nevertheless important from the viewpoint of density and approximation.

\begin{corollary}\label{cor:C_0-semigroup_density}
If a sequence $(\phi_j)_{j = 1}^\infty$ in $C$ converges to $\phi$ in $M^p$, then
\begin{equation*}
	\lim_{j \to \infty} \norm*{S(t)\phi - T(t)\phi_j}_{M^p} = 0
\end{equation*}
holds for all $t \ge 0$.
\end{corollary}

\begin{proof}
This follows by the boundedness of $S(t)$ in view of $T(t)\phi_j = S(t)\phi_j$.
\end{proof}

\section{The Green and Lagrange Identities in the Mild Sense}\label{sec:Lagrange_identity_mild}

\subsection{Bounded Unique Extension of the Bilinear Form $B$}\label{subsec:bounded_extension_B}

In this paper, we introduce the following definition.

\begin{definition}
Let $m^*$ denote the sum of the Lebesgue measure on $[0, r]$ and the Dirac measure at $0$.
We call $m^*$ the \textit{adjoint} of the memory measure $m$.
For each $1 \le q \le \infty$, let $M^{*q}([0, r], (\mathbb{C}^n)\dual)$ be the Banach space given by
\begin{equation*}
	M^{*q}([0, r], (\mathbb{C}^n)\dual) \coloneqq L^q(([0, r], m^*), (\mathbb{C}^n)\dual)
\end{equation*}
endowed with the norm
\begin{equation*}
	\norm*{\argdot}_{M^{*q}} \coloneqq \norm*{\argdot}_{L^q([0, r], m^*)}.
\end{equation*}
Here $([0, r], m^*)$ denotes the measure space $[0, r]$ equipped with the measure $m^*$.
$M^{*q}([0, r], (\mathbb{C}^n)\dual)$ will be abbreviated as $M^{*q}$.
\end{definition}

We observe that the bilinear form $B$ can be expressed in terms of the memory measure $m$ and its adjoint $m^*$.
To see this, we introduce the following notation.

\begin{notation}
For each $\phi \in C$ and $\psi \in C^*$, let $f_\phi \colon [0, r] \to \mathbb{C}^n$ and $f_\psi^* \colon [-r, 0] \to (\mathbb{C}^n)\dual$ be the functions defined by
\begin{align*}
	f_\phi(\zeta)
	&\coloneqq
	\begin{cases}
		\phi(0) & (\zeta = 0), \\
		\int_{-r}^{-\zeta} d\eta(\theta) \mspace{2mu} \phi(\zeta + \theta) & (\zeta \in (0, r]),
	\end{cases} \\[1mm]
	f_\psi^*(\xi)
	&\coloneqq
	\begin{cases}
		\int_{-r}^{\xi} \psi(\xi - \theta) \mspace{2mu} d\eta(\theta) & (\xi \in [-r, 0)), \\
		\psi(0) & (\xi = 0).
	\end{cases}
\end{align*}
\end{notation}

We note that $f_\phi \in L^\infty(([0, r], m^*), \mathbb{C}^n)$ holds because the function $f_\phi$ is Lebesgue measurable and
\begin{equation*}
	\norm*{f_\phi}_{L^\infty([0, r], m^*)}
	= \max\rbra*{\abs*{\phi(0)}, \norm*{f_\phi}_{L^\infty[0, r]}}
	< \infty.
\end{equation*}
By the same reasoning, $f_\psi^* \in L^\infty(([-r, 0], m), (\mathbb{C}^n)\dual)$ also holds.
Then $B(\psi, \phi)$ for $(\psi, \phi) \in C^* \times C$ is expressed as
\begin{align*}
	B(\psi, \phi)
	&= \int_{[-r, 0]} f_\psi^*(\xi)\phi(\xi) \mspace{2mu} m(d\xi) \\
	&= \int_{[0, r]} \psi(\zeta)f_\phi(\zeta) \mspace{2mu} m^*(d\zeta).
\end{align*}

Under these settings, we obtain the following result concerning the extension of the bilinear form $B$.
This is of importance for clarifying the relationship between the mild solutions to \eqref{eq:RFDE} and the adjoint, which is a focus of this section.

\begin{theorem}\label{thm:bounded_unique_extension}
Let $1 < p < \infty$ be given and $q$ be the conjugate exponent of $p$.
Then the bilinear form $B \colon C^* \times C \to \mathbb{C}$ uniquely extends to a bounded bilinear form $\tilde{B} \colon M^{*q} \times M^p \to \mathbb{C}$.
\end{theorem}

\begin{proof}
Let $(\psi, \phi) \in C^* \times C$ be given.
Since $\norm*{f_\phi}_{L^p[0, r]} \le \Var(\eta) \norm*{\phi}_{L^p[-r, 0]}$ holds with the total variation $\Var(\eta)$ of $\eta$ (see \cite{Delfour--Manitius_1980, Nishiguchi_2026}), we have
\begin{align*}
	\norm*{f_\phi}_{L^p([0, r], m^*)}^p
	&= \abs*{\phi(0)}^p + \norm*{f_\phi}_{L^p[0, r]}^p \\
	&\le \abs*{\phi(0)}^p + \Var(\eta)^p \cdot \norm*{\phi}_{L^p[-r, 0]}^p \\
	&\le \max\{1, \Var(\eta)^p\} \cdot \norm*{\phi}_{M^p}^p.
\end{align*}
Therefore, applying H\"older's inequality yields
\begin{align*}
	\abs*{B(\psi, \phi)}
	&\le \int_{[0, r]} \abs*{\psi(\zeta)} \abs*{f_\phi(\zeta)} \mspace{2mu} m^*(d\zeta) \\
	&\le \norm*{\psi}_{M^{*q}} \norm*{f_\phi}_{L^p([0, r], m^*)} \\
	&\le \max\{1, \Var(\eta)^p\}^{1/p} \cdot \norm*{\psi}_{M^{*q}}\norm*{\phi}_{M^p}.
\end{align*}
This shows that the bilinear form $B \colon C^* \times C \to \mathbb{C}$ is bounded with respect to the norms $\norm*{\argdot}_{M^{*q}}$ and $\norm*{\argdot}_{M^p}$.
Since the subset $C^* \times C$ is dense in $M^{*q} \times M^p$, by an argument similar to the unique extension of a bounded linear operator defined on a dense subspace of a normed space to a Banach space, it holds that the bilinear form $B$ uniquely extends to a bounded bilinear form $\tilde{B} \colon M^{*q} \times M^p \to \mathbb{C}$.
\end{proof}

\subsubsection{Remarks on the Unique Bounded Extension}

\begin{remark}\label{rmk:structural_operator}
From the proof of Theorem~\ref{thm:bounded_unique_extension}, the bounded linear operator $C \ni \phi \mapsto f_\phi \in L^p(([0, r], m^*), \mathbb{C}^n)$ uniquely extends to a bounded linear operator
\begin{equation*}
	M^p \ni \phi \mapsto \tilde{f}_\phi \in L^p(([0, r], m^*), \mathbb{C}^n)
\end{equation*}
with
\begin{equation*}
	\norm[\big]{\tilde{f}_\phi}_{L^p([0, r], m^*)}
	\le \max\{1, \Var(\eta)^p\}^{1/p} \cdot \norm*{\phi}_{M^p}
	\quad (\phi \in M^p)
\end{equation*}
by density argument.
This bounded linear operator is related to the so-called \textit{structural operator} introduced in \cite{Bernier--Manitius_1978, Delfour--Manitius_1980} (see Remark~\ref{rmk:Delfour_Manitius} given below).
By using this function $\tilde{f}_\phi$, $\tilde{B}(\psi, \phi)$ is expressed as
\begin{equation*}
	\tilde{B}(\psi, \phi)
	= \int_{[0, r]} \psi(\zeta)\tilde{f}_\phi(\zeta) \mspace{2mu} m^*(d\zeta)
\end{equation*}
for $(\psi, \phi) \in M^{*q} \times M^p$.
\end{remark}

\begin{remark}
By the same reasoning as in Remark~\ref{rmk:structural_operator}, the bounded linear operator $C^* \ni \psi \mapsto f_\psi^* \in L^q(([-r, 0], m), (\mathbb{C}^n)\dual)$ uniquely extends to a bounded linear operator
\begin{equation*}
	M^{*q} \ni \psi \mapsto \tilde{f}_\psi^* \in L^q(([-r, 0], m), (\mathbb{C}^n)\dual)
\end{equation*}
with
\begin{equation*}
	\norm[\big]{\tilde{f}_\psi^*}_{L^q([-r, 0], m)}
	\le \max\{1, \Var(\eta)^q\}^{1/q} \cdot \norm*{\psi}_{M^{*q}}
	\quad (\psi \in M^{*q}).
\end{equation*}
Then $\tilde{B}(\psi, \phi)$ is also expressed as
\begin{equation*}
	\tilde{B}(\psi, \phi)
	= \int_{[-r, 0]} \tilde{f}_\psi^*(\xi)\phi(\xi) \mspace{2mu} m(d\xi)
\end{equation*}
for $(\psi, \phi) \in M^{*q} \times M^p$.
\end{remark}

\begin{remark}\label{rmk:Delfour_Manitius}
Under our framework, the consideration by Delfour and Manitius~\cite{Delfour--Manitius_1980} can be stated as follows.
By showing the boundedness of the linear operator $C \ni \phi \mapsto f_\phi \in L^2([0, r], \mathbb{C}^n)$ with respect to $\norm*{\argdot}_{L^2[-r, 0]}$, they uniquely extend this to the following bounded linear operator
\begin{equation*}
	H \colon L^2([-r, 0], \mathbb{C}^n) \to L^2([0, r], \mathbb{C}^n)
\end{equation*}
and define a bounded linear operator $F \colon L^2([-r, 0], \mathbb{C}^n) \times \mathbb{C}^n \to L^2([0, r], \mathbb{C}^n) \times \mathbb{C}^n$ by
\begin{equation*}
	F\phi \coloneqq \rbra*{H\phi^1, \phi^0}
\end{equation*}
for $\phi = \rbra*{\phi^1, \phi^0} \in L^2([-r, 0], \mathbb{C}^n) \times \mathbb{C}^n$.
Despite the difference in the codomain of $F$, this is precisely the \textit{structural operator} defined in \cite{Delfour--Manitius_1980}.
Delfour and Manitius~\cite{Delfour--Manitius_1980} further defined a bilinear form $\langle \argdot, \argdot \rangle$ on $\rbra*{L^2([-r, 0], \mathbb{C}^n) \times \mathbb{C}^n} \times \rbra*{L^2([-r, 0], \mathbb{C}^n) \times \mathbb{C}^n}$ by
\begin{equation*}
	\langle \psi, \phi \rangle
	\coloneqq \psi^0\phi^0 + \int_{-r}^0 \psi^1(\theta)(H\phi^1)(-\theta) \mspace{2mu} d\theta
\end{equation*}
for $\psi = \rbra*{\psi^1, \psi^0} \in L^2([-r, 0], \mathbb{C}^n) \times \mathbb{C}^n$ and $\phi = \rbra*{\phi^1, \phi^0} \in L^2([-r, 0], \mathbb{C}^n) \times \mathbb{C}^n$, and mentioned that $\langle \psi, \phi \rangle$ is expressed as
\begin{equation*}
	\langle \psi, \phi \rangle
	= \psi(0)\phi(0) + \int_{-r}^0 d\xi \int_{\theta}^0 \psi(\theta - \xi) \cdot d\eta(\theta) \phi(\xi)
\end{equation*}
for $\phi, \psi \in C$.

In this paper, we derive the bilinear form $B \colon C^* \times C \to \mathbb{C}$ from the Green and Lagrange identities for the RFDE~\eqref{eq:RFDE} and uniquely extends this to the bounded bilinear form $\tilde{B} \colon M^{*q} \times M^p \to \mathbb{C}$.
In this respect, our approach is crucially distinct from that of \cite{Delfour--Manitius_1980}.
\end{remark}

\subsection{Mild Solution and Adjoint Semigroup}\label{subsec:mild_sol_adjoint_semigroup}

In this and the subsequent subsection, we fix $1 < p < \infty$ and its conjugate exponent $q$.
Let
\begin{equation*}
	\tilde{B} \colon M^{*q} \times M^p \to \mathbb{C}
\end{equation*}
be the bounded bilinear form whose existence is guaranteed by Theorem~\ref{thm:bounded_unique_extension}.
From Corollary~\ref{cor:C_0-semigroup_density} and Theorem~\ref{thm:bounded_unique_extension}, if $((\psi_j, \phi_j))_{j = 1}^\infty$ is a sequence in $C^* \times C$ converging to $(\psi, \phi)$ in $M^{*q} \times M^p$, then
\begin{equation*}
	\tilde{B}(\psi, S(t)\phi)
	= \lim_{j \to \infty} B(\psi_j, T(t)\phi_j)
\end{equation*}
holds for all $t \ge 0$.
Combining the above equality and Theorem~\ref{thm:adjoint_semigroup} motivates us to introduce the notion of mild solutions to the adjoint equation~\eqref{eq:adjoint_RFDE}, together with the $C_0$-semigroup generated by the corresponding IVP.

In the following, let $\mathcal{L}^1_\mathrm{loc}((-\infty, r], (\mathbb{C}^n)\dual)$ denote the set of all $(\mathbb{C}^n)\dual$-valued locally integrable functions on $(-\infty, r]$.

\begin{definition}
We say that $y \in \mathcal{L}^1_\mathrm{loc}((-\infty, r], (\mathbb{C}^n)\dual)$ is a \textit{mild solution} to \eqref{eq:adjoint_RFDE} if (i) $y$ is defined on $(-\infty, 0]$, and (ii)
\begin{equation*}
	y(t) = y(0) - L^* \int_0^t y^s \mspace{2mu} ds
\end{equation*}
holds for all $t \le 0$.
Here $\int_0^t y^s \mspace{2mu} ds \in C^*$ for $t \le 0$ is defined as
\begin{equation*}
	\rbra*{\int_0^t y^s \mspace{2mu} ds}(\tau)
	= \int_0^t y(s + \tau) \mspace{2mu} ds
	\quad (\tau \in [0, r]),
\end{equation*}
which corresponds to the time-reversed counterpart of the retarded integral.
\end{definition}

Mild solutions to \eqref{eq:adjoint_RFDE} can also be formulated as an IVP by imposing an initial condition of the form
\begin{equation*}
	y(\tau) = \psi(\tau) \quad (\text{$m^*$-a.e.\ $\tau \in [0, r]$}),
\end{equation*}
where $\psi \in M^{*1}$.
We write the above initial condition as $y^0 = \psi$.
Then the IVP is expressed as
\begin{equation}\label{eq:IVP_adjoint_RFDE_mild}
	\left\{
	\begin{aligned}
		y(t) &= \psi(0) - L^* \int_0^t y^s \mspace{2mu} ds & (t \le 0), \\
		y^0 &= \psi
	\end{aligned}
	\right.
\end{equation}
where $\psi \in M^{*1}$.

\begin{remark}
The existence of a unique solution to IVP~\eqref{eq:IVP_adjoint_RFDE_mild} can be established by reducing it to the IVP for $z(t) = y(-t)^T$ $(t \in [-r, \infty))$.
\end{remark}

\begin{definition}
For each $\psi \in M^{*1}$, let $y(\argdot; \psi)$ be the unique solution of \eqref{eq:IVP_adjoint_RFDE_mild}.
We define the family $(S^*(t))_{t \ge 0}$ of linear maps on $M^{*q}$ by
\begin{equation*}
	S^*(t)\psi \coloneqq y(\argdot; \psi)^{-t}
\end{equation*}
for $(t, \psi) \in [0, \infty) \times M^{*q}$.
\end{definition}

Then $(S^*(t))_{t \ge 0}$ constitutes a $C_0$-semigroup.
In the same manner as for Corollary~\ref{cor:C_0-semigroup_density}, if a sequence $(\psi_j)_{j = 1}^\infty$ in $C^*$ converges to $\psi$ in $M^{*q}$, then
\begin{equation*}
	\lim_{j \to \infty} \norm*{S^*(t)\psi - T^*(t)\psi_j}_{M^{*q}} = 0
\end{equation*}
holds for all $t \ge 0$.

Under the above formulation, we obtain the following theorem, which establishes that the $C_0$-semigroup $(S^*(t))_{t \ge 0}$ is ``adjoint" to $(S(t))_{t \ge 0}$ with respect to the extended bilinear form $\tilde{B}$.

\begin{theorem}\label{thm:adjoint_semigroup_mild}
For every $(\psi, \phi) \in M^{*q} \times M^p$,
\begin{equation*}
	\tilde{B}(\psi, S(t)\phi) = \tilde{B}(S^*(t)\psi, \phi) \quad (t \ge 0)
\end{equation*}
holds.
\end{theorem}

\begin{proof}
Let $(\psi, \phi) \in M^{*q} \times M^p$ be given.
Since $C^* \times C$ is dense in $M^{*q} \times M^p$, there is a sequence $((\psi_j, \phi_j))_{j = 1}^\infty$ in $C^* \times C$ converging to $(\psi, \phi)$ in $M^{*q} \times M^p$.
Then
\begin{align*}
	\tilde{B}(\psi, S(t)\phi)
	&= \lim_{j \to \infty} B(\psi_j, T(t)\phi_j) \\
	&= \lim_{j \to \infty} B(T^*(t)\psi_j, \phi_j) \\
	&= \tilde{B}(S^*(t)\psi, \phi)
\end{align*}
holds.
\end{proof}

\subsection{The Lagrange Identity in the Mild Sense}

Theorem~\ref{thm:adjoint_semigroup_mild} strongly suggests that Theorems~\ref{thm:duality_finite_interval} and \ref{thm:Green_identity} hold true in the mild sense by using the extended bilinear form $\tilde{B}$.
The purpose of this subsection is to obtain this extension.

Let $[\alpha, \beta]$ be a closed and bounded interval of $\mathbb{R}$.
In this subsection, we use the following notation.

\begin{notation}
Let $u \in \mathcal{L}^p([\alpha - r, \beta], \mathbb{C}^n)$ and $v \in \mathcal{L}^q([\alpha, \beta + r], (\mathbb{C}^n)\dual)$ be defined on $[\alpha, \beta]$.
\begin{itemize}
\item For each $t \in [\alpha, \beta]$, let $u_t \in M^p$ and $v^t \in M^{*q}$ be defined as
\begin{align*}
	u_t(\theta) &= u(t + \theta) \quad (\text{$m$-a.e.\ $\theta \in [-r, 0]$}), \\
	v^t(\tau) &= v(t + \tau) \quad (\text{$m^*$-a.e.\ $\tau \in [0, r]$}).
\end{align*}
\item For each subinterval $[a, b] \subset [\alpha, \beta]$, let $\int_a^b u_t \mspace{2mu} dt \in C$ and $\int_a^b v^t \mspace{2mu} dt \in C^*$ be defined as
\begin{align*}
	\rbra*{\int_a^b u_t \mspace{2mu} dt}(\theta) &\coloneqq \int_a^b u(t + \theta) \mspace{2mu} dt \quad (\theta \in [-r, 0]), \\
	\rbra*{\int_a^b v^t \mspace{2mu} dt}(\tau) &\coloneqq \int_a^b v(t + \tau) \mspace{2mu} dt \quad (\tau \in [0, r]).
\end{align*}
\end{itemize}
\end{notation}

We first extend the Lagrange identity for the RFDE~\eqref{eq:RFDE} to the mild sense.

\begin{theorem}\label{thm:Lagrange_identity_mild}
Let $u \in \mathcal{L}^p([\alpha - r, \beta], \mathbb{C}^n)$ and $v \in \mathcal{L}^q([\alpha, \beta + r], (\mathbb{C}^n)\dual)$ be given.
If the restrictions $u|_{[\alpha, \beta]}$ and $v|_{[\alpha, \beta]}$ are absolutely continuous, then the function
\begin{equation*}
	[\alpha, \beta] \ni t \mapsto \tilde{B}\rbra*{v^t, u_t} \in \mathbb{C}
\end{equation*}
is absolutely continuous and satisfies
\begin{align*}
	&\frac{d}{dt} \mspace{2mu} \tilde{B}\rbra*{v^t, u_t} \\
	&= \sbra*{\frac{d}{dt} \mspace{2mu} \rbra*{v(t) + L^* \int_\alpha^t v^s \mspace{2mu} ds}} u(t) + v(t) \sbra*{\frac{d}{dt} \mspace{2mu} \rbra*{u(t) - L \int_\alpha^t u_s \mspace{2mu} ds}}
\end{align*}
for a.e.\ $t \in [\alpha, \beta]$.
\end{theorem}

\begin{proof}
\textbf{Step 1.}
In this step, we assume that $u$ and $v$ are continuous functions that are continuously differentiable on $[\alpha, \beta]$.
We show that
\begin{align}
	&\sbra*{\tilde{B}\rbra*{v^t, u_t}}_{t = \alpha}^\beta \notag \\
	&= \sbra*{v(t)u(t)}_{t = \alpha}^\beta + \rbra*{L^*\int_\alpha^\beta v^t \mspace{2mu} dt}u(\beta) - \int_\alpha^\beta \rbra*{L^*\int_\alpha^t v^s \mspace{2mu} ds}\dot{u}(t) \mspace{2mu} dt \notag \\
	&\mspace{30mu} - v(\beta)\rbra*{L\int_\alpha^\beta u_t \mspace{2mu} dt} + \int_\alpha^\beta \dot{v}(t)\rbra*{L\int_\alpha^t u_s \mspace{2mu} ds} \mspace{2mu} dt \label{eq:duality_bilinear_form_retarded}
\end{align}
holds.
From Theorem~\ref{thm:Green_identity}, we have
\begin{align*}
	\sbra*{\tilde{B}\rbra*{v^t, u_t}}_{t = \alpha}^\beta
	&= \sbra*{B\rbra*{v^t, u_t}}_{t = \alpha}^\beta \\
	&= \sbra*{v(t)u(t)}_{t = \alpha}^\beta + \int_\alpha^\beta \rbra*{L^*v^t}u(t) \mspace{2mu} dt - \int_\alpha^\beta v(t)(Lu_t) \mspace{2mu} dt.
\end{align*}
Applying the integration by parts formula yields
\begin{align*}
	\int_\alpha^\beta \rbra*{L^*v^t}u(t) \mspace{2mu} dt
	&= \sbra*{\rbra*{\int_\alpha^t L^*v^s \mspace{2mu} ds}u(t)}_{t = \alpha}^\beta - \int_\alpha^\beta \rbra*{\int_\alpha^t L^*v^s \mspace{2mu} ds}\dot{u}(t) \mspace{2mu} dt \\
	&= \rbra*{\int_\alpha^\beta L^*v^t \mspace{2mu} dt}u(\beta) - \int_\alpha^\beta \rbra*{\int_\alpha^t L^*v^s \mspace{2mu} ds}\dot{u}(t) \mspace{2mu} dt.
\end{align*}
In a similar way, we also have
\begin{equation*}
	\int_\alpha^\beta v(t)(Lu_t) \mspace{2mu} dt
	= v(\beta)\rbra*{\int_\alpha^\beta Lu_t \mspace{2mu} dt} - \int_\alpha^\beta \dot{v}(t)\rbra*{\int_\alpha^t Lu_s \mspace{2mu} ds} \mspace{2mu} dt.
\end{equation*}
Therefore, \eqref{eq:duality_bilinear_form_retarded} is obtained because
\begin{equation*}
	\int_\alpha^t L^*v^s \mspace{2mu} ds = L^*\int_\alpha^t v^s \mspace{2mu} ds, \quad
	\int_\alpha^t Lu_s \mspace{2mu} ds = L\int_\alpha^t u_s \mspace{2mu} ds
\end{equation*}
hold.
\smallskip

\textbf{Step 2.}
In the next step, we will show that \eqref{eq:duality_bilinear_form_retarded} holds for given
\begin{equation*}
	u \in \mathcal{L}^p([\alpha - r, \beta], \mathbb{C}^n)
	\mspace{10mu} \text{and} \mspace{10mu}
	v \in \mathcal{L}^q([\alpha, \beta + r], (\mathbb{C}^n)\dual)
\end{equation*}
that are absolutely continuous on $[\alpha, \beta]$.
For this purpose, we obtain a preliminary equality for the subsequent arguments.

Let $C^1([\alpha, \beta], \mathbb{C}^n)$ be the set of all continuously differentiable functions from $[\alpha, \beta]$ to $\mathbb{C}^n$.
Let $\mu$ be the sum of the Lebesgue measure on $[\alpha - r, \alpha]$ and the Dirac measure at $\alpha$.
Since $C([\alpha - r, \alpha], \mathbb{C}^n)$ and $C^1([\alpha, \beta], \mathbb{C}^n)$ are dense in $L^p(([\alpha - r, \alpha], \mu), \mathbb{C}^n)$ and $\AC([\alpha, \beta], \mathbb{C}^n)$, respectively (see Proposition~\ref{prop:density_C_Mp}), there is a sequence $(u_j)_{j = 1}^\infty$ in $C([\alpha - r, \beta], \mathbb{C}^n)$ such that $u_j$ is continuously differentiable on $[\alpha, \beta]$, and
\begin{gather*}
	\int_{\alpha - r}^\alpha \abs*{u(t) - u_j(t)}^p \mspace{2mu} dt \to 0, \\
	\sup_{t \in [\alpha, \beta]} \abs*{u(t) - u_j(t)} + \int_\alpha^\beta \abs*{\dot{u}(t) - \dot{u}_j(t)} \mspace{2mu} dt \to 0
\end{gather*}
as $j \to \infty$.
In a similar way, there is a sequence $(v_j)_{j = 1}^\infty$ in $C([\alpha, \beta + r], (\mathbb{C}^n)\dual)$ such that $v_j$ is continuously differentiable on $[\alpha, \beta]$, and
\begin{gather*}
	\int_\beta^{\beta + r} \abs*{v(t) - v_j(t)}^q \mspace{2mu} dt \to 0, \\
	\sup_{t \in [\alpha, \beta]} \abs*{v(t) - v_j(t)} + \int_\alpha^\beta \abs*{\dot{v}(t) - \dot{v}_j(t)} \mspace{2mu} dt \to 0
\end{gather*}
as $j \to \infty$.
Then
\begin{align}
	&\sbra*{\tilde{B}\rbra*{v_j^t, (u_j)_t}}_{t = \alpha}^\beta \notag \\
	&= \sbra*{v_j(t)u_j(t)}_{t = \alpha}^\beta + \rbra*{L^*\int_\alpha^\beta v_j^t \mspace{2mu} dt}u_j(\beta) - \int_\alpha^\beta \rbra*{L^*\int_\alpha^t v_j^s \mspace{2mu} ds}\dot{u}_j(t) \mspace{2mu} dt \notag \\
	&\mspace{30mu} - v_j(\beta)\rbra*{L\int_\alpha^\beta (u_j)_t \mspace{2mu} dt} + \int_\alpha^\beta \dot{v}_j(t)\rbra*{L\int_\alpha^t (u_j)_s \mspace{2mu} ds} \mspace{2mu} dt \label{eq:duality_bilinear_form_retarded_j}
\end{align}
holds for every positive integer $j$ from Step 1.
\smallskip

\textbf{Step 3.}
Let $t \in [\alpha, \beta]$ be fixed.
Since
\begin{align*}
	&\norm*{u_t - (u_j)_t}_{M^p}^p \\
	&= \int_{[-r, 0]} \abs*{u(t + \theta) - u_j(t + \theta)}^p \mspace{2mu} m(d\theta) \\
	&\le \int_{\alpha - r}^\alpha \abs*{u(s) - u_j(s)}^p \mspace{2mu} ds + \int_\alpha^\beta \abs*{u(s) - u_j(s)}^p \mspace{2mu} ds + \sup_{s \in [\alpha, \beta]} \abs*{u(s) - u_j(s)}^p,
\end{align*}
we have $\lim_{j \to \infty} (u_j)_t = u_t$ in $M^p$.
In a similar way, we also have $\lim_{j \to \infty} v_j^t = v^t$ in $M^{*q}$.
These relations imply that
\begin{equation*}
	\sbra*{\tilde{B}\rbra*{v^t, u_t}}_{t = \alpha}^\beta
	= \lim_{j \to \infty} \sbra*{\tilde{B}\rbra*{v_j^t, (u_j)_t}}_{t = \alpha}^\beta
\end{equation*}
holds.
Furthermore, since
\begin{align*}
	&\norm*{\int_\alpha^t u_s \mspace{2mu} ds - \int_\alpha^t (u_j)_s \mspace{2mu} ds}_C \\
	&\le \sup_{\theta \in [-r, 0]} \int_\alpha^t \abs*{u(s + \theta) - u_j(s + \theta)} \mspace{2mu} ds \\
	&\le \int_{\alpha - r}^\alpha \abs*{u(s) - u_j(s)} \mspace{2mu} ds + \int_\alpha^t \abs*{u(s) - u_j(s)} \mspace{2mu} ds \\
	&\le r^{1/q} \cdot \rbra*{\int_{\alpha - r}^\alpha \abs*{u(s) - u_j(s)}^p \mspace{2mu} ds}^{1/p} + \int_\alpha^t \abs*{u(s) - u_j(s)} \mspace{2mu} ds,
\end{align*}
we have
\begin{equation*}
	\lim_{j \to \infty} L\int_\alpha^t (u_j)_s \mspace{2mu} ds = L\int_\alpha^t u_s \mspace{2mu} ds
\end{equation*}
by the continuity of $L$.
In a similar way, we also have
\begin{equation*}
	\lim_{j \to \infty} L^*\int_\alpha^t v_j^s \mspace{2mu} ds = L^*\int_\alpha^t v^s \mspace{2mu} ds.
\end{equation*}
Therefore, it holds that the right-hand side of \eqref{eq:duality_bilinear_form_retarded_j} converges to that of \eqref{eq:duality_bilinear_form_retarded} as $j \to \infty$.
Thus, we obtain \eqref{eq:duality_bilinear_form_retarded} for $u \in \mathcal{L}^p([\alpha - r, \beta], \mathbb{C}^n)$ and $v \in \mathcal{L}^q([\alpha, \beta + r], (\mathbb{C}^n)\dual)$ that are absolutely continuous on $[\alpha, \beta]$.
\smallskip

\textbf{Step 4.}
In \cite{Nishiguchi_2026}, it has been shown that the function
\begin{equation*}
	[\alpha, \beta] \ni t \mapsto L\int_\alpha^t u_s \mspace{2mu} ds \in \mathbb{C}^n
\end{equation*}
is absolutely continuous.
The same reasoning of the proof yields that the function
\begin{equation*}
	[\alpha, \beta] \ni t \mapsto L^*\int_\alpha^t v^s \mspace{2mu} ds \in (\mathbb{C}^n)\dual
\end{equation*}
is also absolutely continuous.
Since the identity~\eqref{eq:duality_bilinear_form_retarded} remains true by replacing the interval $[\alpha, \beta]$ with any subinterval $[\alpha', \beta'] \subset [\alpha, \beta]$, the aforementioned absolute continuity implies that the function
\begin{equation*}
	[\alpha, \beta] \ni t \mapsto \tilde{B}\rbra*{v^t, u_t} \in \mathbb{C}
\end{equation*}
is absolutely continuous.
Furthermore, the Lebesgue differentiation theorem (see, e.g., \cite[Theorem~7.10]{Rudin_1987}) combined with the product rule yields
\begin{align*}
	&\frac{d}{dt} \mspace{2mu} \tilde{B}\rbra*{v^t, u_t} \\
	&= \dot{v}(t)u(t) + v(t)\dot{u}(t) \\
	&\mspace{30mu} + \sbra*{\frac{d}{dt} \mspace{2mu} \rbra*{L^* \int_\alpha^t v^s \mspace{2mu} ds}}u(t) + \rbra*{L^* \int_\alpha^t v^s \mspace{2mu} ds} \dot{u}(t) - \rbra*{L^* \int_\alpha^t v^s \mspace{2mu} ds} \dot{u}(t) \\
	&\mspace{30mu} - \dot{v}(t) \rbra*{L\int_\alpha^t u_s \mspace{2mu} ds} - v(t)\sbra*{\frac{d}{dt} \mspace{2mu} \rbra*{L\int_\alpha^t u_s \mspace{2mu} ds}} + \dot{v}(t)\rbra*{L\int_\alpha^t u_s \mspace{2mu} ds} \\
	&= \sbra*{\frac{d}{dt} \mspace{2mu} \rbra*{v(t) + L^* \int_\alpha^t v^s \mspace{2mu} ds}}u(t) + v(t)\sbra*{\frac{d}{dt} \mspace{2mu} \rbra*{u(t) - L\int_\alpha^t u_s \mspace{2mu} ds}}
\end{align*}
for a.e.\ $t \in [\alpha, \beta]$.
\smallskip

These steps complete the proof.
\end{proof}

\begin{corollary}
If $u \in \mathcal{L}^p([\alpha - r, \beta], \mathbb{C}^n)$ and $v \in \mathcal{L}^q([\alpha, \beta + r], (\mathbb{C}^n)\dual)$ are defined on $[\alpha, \beta]$ and satisfies
\begin{equation*}
	u(t) = u(\alpha) + L\int_\alpha^t u_s \mspace{2mu} ds, \mspace{15mu}
	v(t) = v(\beta) - L^*\int_\beta^t v^s \mspace{2mu} ds
\end{equation*}
for $t \in [\alpha, \beta]$, then the function $[\alpha, \beta] \ni t \mapsto \tilde{B}\rbra*{v^t, u_t} \in \mathbb{C}$ is constant.
\end{corollary}

\begin{proof}
Since $v(\alpha) = v(\beta) - L^*\int_\beta^\alpha v^s \mspace{2mu} ds$ holds by the assumption, we have
\begin{align*}
	v(t)
	&= v(\beta) - L^*\int_\beta^t v^s \mspace{2mu} ds \\
	&= v(\alpha) + L^*\int_\beta^\alpha v^s \mspace{2mu} ds - L^*\int_\beta^t v^s \mspace{2mu} ds \\
	&= v(\alpha) - L^*\int_\alpha^t v^s \mspace{2mu} ds.
\end{align*}
Therefore, applying Theorem~\ref{thm:Lagrange_identity_mild} yields
\begin{equation*}
	\frac{d}{dt} \mspace{2mu} \tilde{B}\rbra*{v^t, u_t} = 0
\end{equation*}
for a.e.\ $t \in [\alpha, \beta]$.
This leads to the conclusion.
\end{proof}

Theorem~\ref{thm:Lagrange_identity_mild} leads us to the following Green identity in the mild sense.

\begin{corollary}\label{cor:Green_identity_mild}
Let $u \in \mathcal{L}^p([\alpha - r, \beta], \mathbb{C}^n)$ and $v \in \mathcal{L}^q([\alpha, \beta + r], (\mathbb{C}^n)\dual)$ be given.
If the restrictions $u|_{[\alpha, \beta]}$ and $v|_{[\alpha, \beta]}$ are absolutely continuous, then
\begin{align*}
	&\int_\alpha^\beta \cbra*{\sbra*{\frac{d}{dt} \mspace{2mu} \rbra*{v(t) + L^* \int_\alpha^t v^s \mspace{2mu} ds}} u(t) + v(t) \sbra*{\frac{d}{dt} \mspace{2mu} \rbra*{u(t) - L \int_\alpha^t u_s \mspace{2mu} ds}}} \mspace{2mu} dt \\
	&= \sbra*{\tilde{B}\rbra*{v^t, u_t}}_{t = \alpha}^\beta
\end{align*}
holds.
\end{corollary}

\begin{proof}
The conclusion is obtained from Theorem~\ref{thm:Lagrange_identity_mild} by using the fundamental theorem of calculus (see, e.g., \cite[Theorem~7.20]{Rudin_1987}).
\end{proof}

\section{Conclusion}

In this paper, we have established the Green identity and the Lagrange identity for a given autonomous linear RFDE, thereby deriving the associated bilinear form $B$ (the Hale bilinear form) and the adjoint equation.
Consequently, the structural origin of the Hale bilinear form, which had remained elusive since its \textit{ad hoc} introduction by Hale, has been clarified as a mathematical necessity within the framework of theory of differential equations.
Furthermore, through the dense extension of this bilinear form to the $L^p$-framework utilizing the memory measure $m$, we have extended the Green and Lagrange identities to the mild sense.

The adjoint theory established in this study is expected to provide a more explicit and computable alternative foundation for the normal form theory and invariant manifold theory developed by Faria and Magalh\~aes~\cite{Faria_Magalhaes_1995_BT, Faria_Magalhaes_1995_Hopf}, in contrast to the abstract sun-star calculus framework established by Diekmann and his collaborators (see \cite{Diekmann--vanGils--Lunel--Walther_1995}).

Moreover, our framework opens up new possibilities for extending the theory to non-autonomous systems, neutral functional differential equations, and differential equations with delay and spatial structures.
In particular, since linearizing a nonlinear RFDE along a limit cycle or a homoclinic orbit yields a non-autonomous linear RFDE, the extension of our approach to non-autonomous systems is of paramount importance.
It is anticipated to provide a rigorous and explicit real-analytic bedrock not only for geometric investigations of global bifurcations such as exponential dichotomies by Lin~\cite{Lin_1986}, but also for recent developments in the phase reduction theory for delayed systems spanning from delay differential equations studied by Novi\v{c}enko and Pyragas~\cite{Novicenko--Pyragas_2012} as well as Kotani et al.~\cite{Kotani--Yamaguchi--Ogawa--Jimbo--Nakao--Ermentrout_2012} to reaction--diffusion systems with delay investigated by Ozawa and Kawamura~\cite{Ozawa--Kawamura_2026}.

\section*{Acknowledgements}

This work was supported by JSPS Grant-in-Aid for Young Scientists Grant Number JP23K12994.

\end{document}